\documentclass[11pt]{article}
\usepackage{amsmath,amsthm,amssymb,bm,color,xspace,booktabs}

\usepackage[margin=1in]{geometry}
\usepackage{xcolor}
\usepackage[colorlinks=true]{hyperref}
\usepackage{algorithm}
\usepackage{algorithmic}
\usepackage[capitalise]{cleveref}
\usepackage{comment}
\usepackage{url}

\newtheorem{theorem}{Theorem}[section]
\newtheorem{lemma}[theorem]{Lemma}
\newtheorem{corollary}[theorem]{Corollary}
\newtheorem{proposition}[theorem]{Proposition}
\newtheorem{claim}{Claim}

\theoremstyle{plain}
\newtheorem{definition}[theorem]{Definition}

 \usepackage{multirow}

\newcommand{\A}{\mathbf{A}}
\newcommand{\B}{\mathbf{B}}
\newcommand{\M}{\mathbf{M}}
\newcommand{\X}{\mathbf{X}}
\newcommand{\Y}{\mathbf{Y}}
\newcommand{\D}{\mathbf{D}}
\newcommand{\tA}{\widetilde{\mathbf{A}}}
\newcommand{\sa}{\mathbf{a}}
\newcommand{\x}{\mathbf{x}}
\newcommand{\y}{\mathbf{y}}
\newcommand{\R}{\mathbb{R}}
\newcommand{\memo}[1]{\textbf{ #1 }}

\title{Online Spectral Approximation in Random Order Streams} 

\author{
 Masataka Gohda\\ The University of Tokyo \\  \texttt{masataka\_goda@mist.i.u-tokyo.ac.jp}
  \and
  Naonori Kakimura\thanks{Supported by JST ERATO Grant Number JPMJER1201, Japan, and by JSPS KAKENHI Grant Numbers JP17K00028 and JP18H05291.}\\ Keio University\\  \texttt{kakimura@math.keio.ac.jp}
}

\begin{document}

\maketitle

\begin{abstract}
This paper studies spectral approximation for a positive semidefinite matrix in the online setting.
It is known in [Cohen \emph{et al}. APPROX 2016] that we can construct a spectral approximation of a given $n\times d$ matrix in the online setting if an additive error is allowed.
In this paper, we propose an online algorithm that avoids an additive error with the same time and space complexities as the algorithm of Cohen \emph{et al.}, and provides a better upper bound on the approximation size when a given matrix has small rank.
In addition, we consider the online random order setting where a row of a given matrix arrives uniformly at random.
In this setting, we propose time and space efficient algorithms to find a spectral approximation.
Moreover, we reveal that a lower bound on the approximation size in the online random order setting is $\Omega (d \epsilon^{-2} \log n)$, which is larger than the one in the offline setting by an $\mathrm{O}\left( \log n \right)$ factor.
%
%
\end{abstract}

\section{Introduction}
\label{section:introduction}
\emph{Spectral sparsification} is to compress the Laplacian matrix of a dense graph to the one of a sparse graph maintaining its quadratic form for an arbitrary vector. 
It was introduced by Spielman and Teng~\cite{DBLP:conf/stoc/SpielmanT04} as a generalization of cut sparsification, and they presented a nearly-linear-time algorithm for spectral sparsification.
Since then algorithms for spectral sparsification have become faster and more refined~\cite{DBLP:conf/stoc/BatsonSS09, DBLP:conf/stoc/Lee017, DBLP:journals/siamcomp/SpielmanT11} and brought a new paradigm to numerical linear algebra and spectral graph theory.
In particular, it led to efficient algorithms for several problems such as linear systems in symmetric diagonally dominant matrices~\cite{DBLP:conf/stoc/CohenKMPPRX14, DBLP:conf/focs/KoutisMP11}, maximum $s$-$t$ flow problems~\cite{DBLP:conf/soda/KelnerLOS14, DBLP:conf/soda/Peng16}, and linear programming~\cite{DBLP:conf/focs/LeeS14}.
Spectral sparsification has been generalized to positive semidefinite~(PSD) matrices~\cite{DBLP:conf/innovations/CohenLMMPS15, DBLP:conf/focs/LiMP13, DBLP:journals/corr/abs-1008-0587}, which we call \emph{spectral approximation} to distinguish from spectral sparsification.
For a matrix $\A$ in $\mathbb{R}^{n \times d}$, a matrix $\tA\in \mathbb{R}^{n' \times d}$ is called a \emph{$(1\pm \epsilon)$-spectral approximation for $\A$} if, for every $\x\in \R^d$, $\x^\top \tA^\top \tA  \x$ approximates $\x^\top \A^\top \A\x$ within a factor of $1\pm \epsilon$.
The number of rows in $\widetilde{\mathbf{A}}$, $n'$, is called the {\it approximation size}.
It is known that there exists an algorithm that returns a $(1\pm \epsilon)$-spectral approximation with approximation size $\mathrm{O} (d / \epsilon^2)$~\cite{DBLP:conf/stoc/BatsonSS09, DBLP:conf/stoc/Lee017}, while the approximation size is $\Omega (d / \epsilon^2)$ in the worst case~\cite{DBLP:conf/stoc/BatsonSS09}.


\subsection{Our Results}

Recently, spectral approximation has been studied in restrictive settings such as the semi-streaming setting and the online setting~\cite{DBLP:conf/approx/CohenMP16, DBLP:conf/soda/KyngPPS17}. 
In the line of research, this paper focuses on the online setting. 
In the online setting, a matrix $\A$ is not known in advance, and each row of $\A$ arrives one-by-one.
Each time we receive a row of $\A$, we decide irrevocably whether to choose the row for a resulting spectral approximation or not and cannot discard or reweight it later.
Cohen \emph{et al.}~\cite{DBLP:conf/approx/CohenMP16} proposed a simple algorithm for an $(\epsilon, \delta)$-spectral approximation,
where an \emph{$( \epsilon, \delta )$-spectral approximation} is a matrix $\widetilde{\mathbf{A}}$ such that $(1-\epsilon)\mathbf{A}^\top \mathbf{A} - \delta \mathbf{I} \preceq \widetilde{\mathbf{A}}^\top\widetilde{\mathbf{A}} \preceq (1+ \epsilon) \mathbf{A}^\top \mathbf{A} + \delta \mathbf{I}$. 
The approximation size is shown to be $\mathrm{O} \left( d \epsilon^{-2} \log d \log \left( \epsilon \| \mathbf{A} \|_2^2 / \delta \right) \right)$. 
They further improved the approximation size to  $\mathrm{O} \left( d \epsilon^{-2} \log \left( \epsilon \| \mathbf{A} \|_2^2 / \delta \right) \right)$ by an $\mathrm{O}(\log d)$ factor with the aid of a method to obtain linear-sized approximation~\cite{DBLP:conf/stoc/Lee017}.
This is asymptotically optimal in the sense that no algorithm based on row sampling can obtain an $( \epsilon, \delta )$-spectral approximation with $\mathrm{o} \left( d \epsilon^{-2} \log \left( \epsilon \| \mathbf{A} \|_2^2 / \delta \right) \right)$ rows~\cite{DBLP:conf/approx/CohenMP16}.



The main contributions of this paper are threefold.
First, we revisit the online spectral approximation algorithm by Cohen \emph{et al.}~\cite{DBLP:conf/approx/CohenMP16}, and remove the additional parameter $\delta$ to obtain a $(1 \pm \epsilon)$-spectral approximation.
Second, we consider the case when each row arrives uniformly at random for the online spectral approximation problem, and propose a fast and memory-efficient algorithm that achieves a $(1 \pm \epsilon)$-spectral approximation.
Finally, we reveal a lower bound on the approximation size in the online random order setting.
Let us describe our results in more detail.

\paragraph{Online setting.}
The online spectral approximation algorithm by Cohen \emph{et al.}~\cite{DBLP:conf/approx/CohenMP16} is simple and optimal with respect to the approximation size, but it entails the additive error $\delta$. 
It is not difficult to modify their algorithm to the one for finding a $(1 \pm \epsilon)$-spectral approximation by setting $\delta = \epsilon \min_i \left(\sigma_{\rm min} (\A_i)^2 \right)$, where  $\A_i$ is the matrix composed of the first $i$ rows in $\mathbf{A}$ and $\sigma_{\rm min}(\A_i)$ is the smallest non-zero singular value of $\A_i$.
However, this requires us to know some estimation of $\min_i \left(\sigma_{\rm min} (\A_i)^2 \right)$ beforehand. 
Our first result is to present spectral approximation algorithms without such prior information.
We propose two algorithms~(\cref{theorem:approximation_size_in_online_row_sampling} and \cref{theorem:approximation_size_in_online_row_sampling2}) by analogy with Cohen \emph{et al.}~\cite{DBLP:conf/approx/CohenMP16}.

To state our results, we denote 
\begin{align}
        \mu (\A) \overset{\mathrm{def}}{=} \frac{\| \mathbf{A} \|^2_2}{\min_{1 \le i \le n} \left(\sigma_{\rm min} (\A_i)^2 \right)}. 
\end{align}
Note that $\mu (\A)$ may be 1, e.g., when $\A$ is the identity matrix.
We say that an event with $\A \in \mathbb{R}^{n \times d}$ happens \emph{with high probability} if it happens with probability at least $1 - 1 / \mathrm{poly} (d)$

\begin{theorem}
\label{theorem:approximation_size_in_online_row_sampling}
    Let $\mathbf{A} \in \mathbb{R}^{n \times d}$ be a matrix of rank $r$, and $\epsilon \in (0,1/2]$ be an error parameter.
    Then, in the online setting, we can construct with high probability a $(1 \pm \epsilon)$-spectral approximation for $\mathbf{A}$ whose approximation size is $\mathrm{O} \left( \left( r \log \mu (\A) + r + \log d \right) \epsilon^{-2} \log d \right)$.
\end{theorem}

\begin{theorem}
\label{theorem:approximation_size_in_online_row_sampling2}
    Let $\mathbf{A} \in \mathbb{R}^{n \times d}$ be a matrix of rank $r$, and $\epsilon \in (0,1)$ be an error parameter.
    Then, in the online setting, we can construct a $(1 \pm \epsilon)$-spectral approximation for $\mathbf{A}$
    whose approximation size is $\mathrm{O} \left( r \epsilon^{-2} \log \mu (\A) + r \epsilon^{-2} \right)$ in expectation.
\end{theorem}

We remark that, using the same instance as in Theorem 5.1 of Cohen \emph{et al.}~\cite{DBLP:conf/approx/CohenMP16}, it turns out that the approximation size in \cref{theorem:approximation_size_in_online_row_sampling2} is asymptotically optimal.
That is, no algorithm can obtain a $( 1\pm \epsilon)$-spectral approximation with $\mathrm{o} \left( r \epsilon^{-2} \log \mu (\A) + r \epsilon^{-2} \right)$ rows~(see \cref{thm:onlein_lb}).

Let us compare our algorithms with $( \epsilon, \delta )$-spectral approximation algorithms in \cite{DBLP:conf/approx/CohenMP16}.
If we set $\delta = \epsilon \min_i \sigma_{\rm min} (\A_i)^2$, then their online algorithms return a $(1\pm \epsilon)$-spectral approximation whose sizes are 
$\mathrm{O}\left( d \epsilon^{-2} \log d \log \mu (\A) \right)$ and $\mathrm{O}\left( d \epsilon^{-2} \log  \mu (\A) \right)$, respectively.
Therefore, \cref{theorem:approximation_size_in_online_row_sampling} and \cref{theorem:approximation_size_in_online_row_sampling2} give better upper bounds on the approximation size as $r \ll d$. 
Note that the approximation size of \cite{DBLP:conf/approx/CohenMP16} always depends on $d$ due to the regularizing factor $\delta$.
Note also that the running time and the space complexity are the same.

The framework of our first algorithm is similar to Cohen \emph{et al.}~\cite{DBLP:conf/approx/CohenMP16}: Each time we receive a row, we compute a score of the arriving row with a matrix we have at the moment, that measures the importance of the arriving row.
Then we decide whether to sample the row or not based on the score.
Cohen \emph{et al.}~\cite{DBLP:conf/approx/CohenMP16} used the \emph{online ridge leverage score}, assuming that the current matrix~(together with $\delta\mathbf{I}$) is nonsingular.
On the other hand, we introduce a new score called a \emph{relative leverage score}, that handles a singular current matrix directly.
To analyze relative leverage scores for singular matrices, we exploit the pseudo-determinant~\cite{pseudo18, pseudo14}, which is of independent interest.

Cohen \emph{et al.}~\cite{DBLP:conf/approx/CohenMP16} improved the approximation size by an $\mathrm{O}\left( \log d \right)$ factor, based on a technique to obtain linear-sized approximation introduced in \cite{DBLP:conf/focs/LeeS15a}.
The linear-sized approximation technique was originally developed in~\cite{DBLP:conf/stoc/BatsonSS09}, and later made randomized to obtain a faster algorithm~\cite{DBLP:conf/stoc/Lee017}.
We combine the analysis in \cref{theorem:approximation_size_in_online_row_sampling} with the randomized one to obtain \cref{theorem:approximation_size_in_online_row_sampling2}.
Since the proof to combine them is almost identical to that of Cohen \emph{et al.}~\cite{DBLP:conf/approx/CohenMP16}, it is given in \cref{section:appendix}.

\paragraph{Online random order setting.}
The approximation sizes in \cref{theorem:approximation_size_in_online_row_sampling} and \cref{theorem:approximation_size_in_online_row_sampling2} depend on $\mu \left( \mathbf{A} \right)$.
Thus when $\mu \left( \mathbf{A} \right)$ exceeds $\mathrm{poly}\left( n \right)$, the approximation size can be $\omega \left( r~\mathrm{poly} (\log n, \epsilon) \right)$. 
In fact, there exists a matrix $\mathbf{A}$ such that we have to sample all the rows in $\mathbf{A}$ to construct a $(1\pm\epsilon)$-spectral approximation in the online setting.
Moreover, in our algorithms as well as algorithms in \cite{DBLP:conf/approx/CohenMP16}, we need to compute a Moore-Penrose pseudo-inverse to evaluate the score in each iteration, and thus the running time is not efficient.
In fact, even if we exploit the Sharman-Morrison formula for the Moore-Penrose pseudo-inverses, it takes $\mathrm{O}( n d^2 )$ time in total. 
Our second and third contributions are to study the online random order setting to break these difficulties.

In the online random order setting, each row in an input matrix $\A$ comes in exactly once according to a certain random permutation in addition to the online setting. 
More formally, we are given a family of row vectors $\X$.
Let $\mathcal{A}(\mathbf{X})$ be a discrete uniform distribution whose element is a matrix obtained by permuting vectors in $\mathbf{X}$.
We note that $\mathcal{A}(\mathbf{X})$ has $n!$ elements.
Then the \emph{online random order setting} means that an input matrix is a random variable $\mathbf{A} \sim \mathcal{A}(\mathbf{X})$, and is given as a stream of rows in the online setting.
Algorithms in the random order setting have been analyzed for several problems such as frequency moment estimation~\cite{DBLP:conf/icalp/BravermanVWY18}, computation of the median~\cite{DBLP:conf/soda/ChakrabartiJP08}, and approximation of maximum matching in a graph~\cite{DBLP:conf/soda/KapralovKS14}, and it is often shown that randomness breaks the worst-case complexity of the online setting.

In the online random order setting, we propose a fast and memory-efficient algorithm that returns a $(1\pm\epsilon)$-spectral approximation.
The approximation size is $\mathrm{O}\left( d \epsilon^{-2} \log n \log d \right)$ for almost all rows' permutations, which is independent of $\mu (\A)$. We here denote the number of nonzero entries of a matrix $\A$ by $\mathrm{nnz} (\A)$.

\begin{theorem}
\label{theorem:improved_scaled_sampling}
    Let $\epsilon \in (0, 1/2]$ be an error parameter, and $\X$ be a family of $n$ row vectors in $\mathbb{R}^d$.
    Then there exists an algorithm in the online random order setting such that $\A \sim \mathcal{A}(\X)$ satisfies the following with high probability: The algorithm returns a $(1 \pm \epsilon)$-spectral approximation for $\mathbf{A}$ with $\mathrm{O}\left( d \epsilon^{-2} \log n \log d \right)$ rows with high probability.
    It consumes $\mathrm{O}\left( \mathrm{nnz}(\mathbf{A}) \log n + \left( d^\omega + d^2 \log n \right) \log^2 n \right.$ $\left. \log d \right)$ time and stores $\mathrm{O}\left( d \log n \log d \right)$ rows as the working memory and $\mathrm{O}\left( d \epsilon^{-2} \log n \log d \right)$ rows as the output memory.
\end{theorem}

We note that there are two kinds of randomness: random permutation in an input and random sampling in algorithms.

For the proof of \cref{theorem:improved_scaled_sampling}, we first present a simpler algorithm with less efficient time and space complexity~(\cref{theorem:scaled_sampling}).
The idea of our algorithm is simple.
Recall that our algorithm in \cref{theorem:approximation_size_in_online_row_sampling} computes the relative leverage score with a current matrix, and samples an arriving row based on that score.
The time-consuming part is to compute a pseudo-inverse to obtain the relative leverage score each time a current matrix is updated.
In the proposed algorithm, we keep using the same matrix to compute the relative leverage scores for consecutive rows in a batch, which reduces the number of computing a pseudo-inverse.
The correctness of the algorithm consists of three parts: (i) the output is a $(1 \pm \epsilon)$-spectral approximation, (ii) the approximation size is bounded, and (iii) the algorithm runs in desired time and space.
The first part (i) can be shown with a matrix martingale.
The proof is similar to Cohen \emph{et al.}~\cite{DBLP:conf/approx/CohenMP16}, but we need careful analysis due to the fact that we use the relative leverage score with a matrix~(which depends on results of previous samples).
Note that (i) holds independently of row permutations, that is, (i) holds for any $\A \sim \mathcal{A}(\X)$.
The randomness of row permutations is exploited to prove (ii) and (iii).
We make use of the result in~\cite{DBLP:conf/innovations/CohenLMMPS15} that a matrix obtained by sampling rows uniformly at random from an input matrix gives a good approximation of leverage scores.
Owing to this fact, we prove that the number of computing the Moore-Penrose pseudo-inverses used for the relative leverage scores is reduced to $\mathrm{O}(\log n)$, keeping the approximation size small.
This simple algorithm, with further observations to reduce time complexity, yields \cref{theorem:improved_scaled_sampling}.
Moreover, we also prove that the simple algorithm, together with a semi-streaming algorithm~\cite{DBLP:conf/soda/KyngPPS17}, can reduce the working memory space to $\mathrm{O} \left( d \log d \right)$ rows, which does not depend on $\epsilon$ and $n$ (See \cref{theorem:improved_scaled_sampling_less_memory}).

We remark that the approximation size can be improved to $\mathrm{O}\left( d \epsilon^{-2} \log n \right)$ using the algorithm in \cref{theorem:approximation_size_in_online_row_sampling2}, although the running time and the space complexity are much less efficient (See \cref{theorem:optimal_random_order_streams}). 


\paragraph{Lower bound in the online random order setting.}

On the other hand, we obtain a lower bound for the online random order setting. 
We prove that any algorithm that selects rows in the online random order setting and returns a $(1 \pm \epsilon)$-spectral approximation must sample $\Omega \left( d \epsilon^{-2} \log n \right)$ rows with high probability, and thus the algorithm in \cref{theorem:optimal_random_order_streams} is asymptotically optimal in terms of the approximation size.
Since the lower bound on the approximation size in the offline setting is $\Omega (d\epsilon^{-2})$, the online random order setting suffers an additional $\log n$ factor.

\begin{theorem}
\label{theorem:optimal}
  Let $\epsilon \in (0, 1)$ be an error parameter. 
  Let $R$ be an algorithm that samples rows in the online random order setting and returns a $(1 \pm \epsilon)$-spectral approximation. 
  Then there exists a family of row vectors $\mathbf{X}$ in $\mathbb{R}^d$ such that $R$ with an input $\mathbf{A}\sim \mathcal{A}(\X)$ returns $\Omega \left( d \epsilon^{-2} \log n \right)$ rows with high probability.
\end{theorem}

We define the worst instance to be the incidence matrix of a graph on $d$ vertices such that every pair of vertices has $n/\binom{d}{2}$ parallel edges.
Then the key ingredient is that there exists an integer $D$ such that, if we sample $D$ rows uniformly at random from the instance, then the corresponding matrix is a $(1\pm \epsilon)$-spectral approximation for a weighted complete graph with high probability.
Since a weighted complete graph on $d$ vertices requires $\Omega (d\epsilon^{-2})$ rows for a $(1\pm \epsilon)$-spectral approximation, this implies that we need to sample $\Omega (d\epsilon^{-2})$ rows while $D$ rows arrive.
By dividing the rows of $\A$ into $\mathrm{O}(\log n)$ parts with a geometric series $D, 2D, 4D, \dots$, this yields \cref{theorem:optimal}.

\subsection{Related work.}

The difficulty of spectral approximation in restrictive settings such as the semi-streaming setting and the online setting lies in that the probability of sampling a row becomes dependent on which rows are sampled so far.
In the standard spectral approximation~(without any restriction), 
we are given all the rows of a matrix $\mathbf{A}$ in advance.
Then we can sample each row by setting a sampling probability with $\A$.
The matrix Chernoff bound provides an exponential tail bound, which ensures that the output $\widetilde{\mathbf{A}}$ is a $(1\pm\epsilon)$-spectral approximation.
However, in the semi-streaming or the online setting,
the probability that an algorithm returns a spectral approximation is no longer bounded by the matrix Chernoff bound due to the dependencies of the sampled rows. 
We need to analyze carefully how the output $\widetilde{\mathbf{A}}$ is constructed in the process.
Cohen \emph{et al.}~\cite{DBLP:conf/approx/CohenMP16} and Kyng \emph{et al.}~\cite{DBLP:conf/soda/KyngPPS17} made use of a matrix martingale and its exponential tail bound, in which we are allowed to have mutually dependent random variables.
The analysis with a matrix martingale is also used in other papers such as \cite{DBLP:conf/focs/CohenKKPPRS18, DBLP:conf/focs/KyngS16, DBLP:conf/focs/KyngS18}.
We also employ it in our settings.

In the dynamic setting, we aim to maintain a spectral approximation under row insertions and deletions.
For a special case where $\A$ is the Laplacian matrix, Abraham \emph{et al}. \cite{DBLP:conf/focs/AbrahamDKKP16} showed that we can construct a $(1\pm \epsilon)$-spectral approximation such that the amortized update time per insertion or deletion is $\mathrm{O} \left( \mathrm{poly} \left( \log d, \epsilon \right) \right)$, which was later de-amortized by \cite{DBLP:conf/soda/BernsteinFH19}.
Their algorithms maintain spanners in a graph, which is a different approach from our algorithms based on sampling rows in the online setting.
When $\A$ is the Laplacian matrix, 
we can also obtain a spectral approximation in the dynamic semi-streaming setting~\cite{DBLP:conf/focs/KapralovLMMS14,DBLP:journals/corr/abs-1903-12150}.

There exist other research directions such as a spectral sketch~\cite{DBLP:conf/innovations/AndoniCKQWZ16, DBLP:conf/soda/JambulapatiS18} and spectral sparsification for a generalization of undirected graphs~\cite{DBLP:journals/corr/abs-1905-01495, DBLP:conf/focs/CohenKKPPRS18, DBLP:conf/stoc/CohenKPPRSV17, DBLP:conf/soda/SomaY19}.
For a matrix $\mathbf{A} \in \mathbb{R}^{d \times d}$, a \emph{spectral sketch} is a function $f$ such that $(1 - \epsilon) \mathbf{x}^\top \mathbf{A} \mathbf{x} \preceq f \left( \mathbf{x} \right) \preceq (1 + \epsilon) \mathbf{x}^\top \mathbf{A} \mathbf{x}$ for every vector $\mathbf{x}$.
If $\mathbf{A}$ is the Laplacian matrix, it is known that there is a nearly-linear time algorithm which returns a spectral sketch with $\mathrm{O} \left( d / \epsilon \right)$ bits, which is better than $\Omega \left( d / \epsilon^2 \right)$ bits of a $(1\pm \epsilon)$-spectral sparsifier~\cite{DBLP:conf/soda/JambulapatiS18}.
Recently, spectral sparsification has been extended to directed graphs~\cite{DBLP:conf/focs/CohenKKPPRS18, DBLP:conf/stoc/CohenKPPRSV17} and hypergraphs~\cite{DBLP:journals/corr/abs-1905-01495, DBLP:conf/soda/SomaY19}. 
\subsection{Organization}

The paper is organized as follows.
In \cref{section:preliminaries}, we define the leverage score and the relative leverage score, and discuss their properties.
In \cref{section:original}, we give a $(1 \pm \epsilon)$-spectral approximation algorithm in the online setting~(\cref{theorem:approximation_size_in_online_row_sampling}). 
In \cref{section:main}, we develop algorithms in the online random order setting~(\cref{theorem:improved_scaled_sampling}). 
In \cref{section:optimality}, we prove a lower bound on the approximation size in the online random order setting~(\cref{theorem:optimal}). 
In \cref{section:appendix}, we present asymptotically optimal algorithms in the online setting and the online random order setting based on techniques by Cohen \emph{et al}. \cite{DBLP:conf/approx/CohenMP16} (\cref{theorem:approximation_size_in_online_row_sampling2}, \cref{theorem:optimal_random_order_streams}). 
In this paper, we often use some facts about positive semidefinite ordering of matrices in the proof, which are discussed in \cref{app:pseudoinverse} for completeness.

\section{Preliminaries}
\label{section:preliminaries}




\subsection{Spectral Approximation}

We define the {\it $(1 \pm \epsilon)$-spectral approximation}.
Recall that, for two symmetric matrices $\A$, $\B$, we denote $\A \preceq \B$ if $\B-\A$ is a positive semidefinite~({\rm PSD}) matrix.
\begin{definition} [Spectral Approximation]
  Let $\mathbf{A} \in \mathbb{R}^{n \times d}$ be a matrix and $\epsilon \in (0,1)$ be an error parameter.
  We say that $\widetilde{\mathbf{A}} \in \mathbb{R}^{m \times d}$ is a $(1 \pm \epsilon)$-spectral approximation for $\mathbf{A}$ if
  \begin{displaymath}
    (1-\epsilon)\mathbf{A}^\top \mathbf{A} \preceq \widetilde{\mathbf{A}}^\top \widetilde{\mathbf{A}} \preceq (1+\epsilon)\mathbf{A}^\top \mathbf{A}.
  \end{displaymath}
\end{definition}
Notice that, by the Courant-Fischer theorem, each eigenvalue of $\widetilde{\mathbf{A}}^\top \widetilde{\mathbf{A}}$ approximates the corresponding one of $\A^\top\A$ within a factor of $1 \pm \epsilon$.

For an edge-weighted graph $G$ with $d$ vertices and $n$ edges,  
we denote the incidence matrix for $G$ by $\mathbf{B}_G \in \mathbb{R}^{n \times d}$, and the Laplacian matrix of $G$ by $\mathbf{L}_G \in \mathbb{R}^{d \times d}$.
Thus $\mathbf{L}_G = \mathbf{B}^\top_G\mathbf{B}_G$ holds.
A  $(1 \pm \epsilon)$-spectral approximation for $\mathbf{B}_G$ is called a \emph{$(1\pm\epsilon)$-spectral sparsifier of $G$} .

In spectral approximation algorithms, the {\it leverage score} plays a major role, which is defined as below.
For a matrix $\A$, the $i$-th row is denoted by $\mathbf{a}^\top_i$, and the Moore-Penrose pseudo-inverse of $\A$ is denoted by $\A^+$.

\begin{definition} [Leverage Score]
  Let $\mathbf{A} \in \mathbb{R}^{n \times d}$ be a matrix.
  For $i\in \{1,2,\dots, n\}$, the \emph{leverage score $\tau_i(\mathbf{A})$} is defined to be 
  \begin{displaymath}
    \tau_i(\mathbf{A}) \overset{\mathrm{def}}{=}
    \mathbf{a}_i^\top \left( \mathbf{A}^\top \mathbf{A} \right)^+ \mathbf{a}_i.
  \end{displaymath}
\end{definition}
When an input matrix is the Laplacian matrix of a graph, the leverage score with respect to an edge $e = (a,b)$ is equivalent to the effective resistance between $a$ and $b$ in the graph~(see e.g., \cite{DBLP:journals/siamcomp/SpielmanS11}). We show properties of the leverage score.
\begin{lemma}
\label{lemma:property_leverage_score}
  For a matrix $\mathbf{A} \in \mathbb{R}^{n \times d}$, we have the following properties:
  \begin{enumerate}
      \item[\textup{(i)}] $0 \le \tau_i \left( \mathbf{A} \right) \le 1$ for any $i\in \{1,2,\dots, n\}$.
      \item[\textup{(ii)}] $\sum_{i=1}^n \tau_i \left( \mathbf{A} \right) = \mathrm{rank} \left( \mathbf{A} \right)$.
  \end{enumerate}
\end{lemma}
\begin{proof}
\textbf{(i)}
Since $\mathbf{x}^\top \left( \mathbf{A}^\top \mathbf{A} \right)^+ \mathbf{x} \le \mathbf{x}^\top \left( \mathbf{a}_i \mathbf{a}_i^\top \right)^+ \mathbf{x}$ holds for all $\mathbf{x} \in \mathrm{Im} \left( \mathbf{a}_i \mathbf{a}_i^\top \right)$, we have
\begin{displaymath}
    \tau_i \left( \mathbf{A} \right) \le \mathbf{a}_i^\top \left( \mathbf{a}_i \mathbf{a}_i^\top \right)^+ \mathbf{a}_i = \mathrm{tr} \left( \mathbf{a}_i^\top \left( \mathbf{a}_i \mathbf{a}_i^\top \right)^+ \mathbf{a}_i \right) = \mathrm{tr} \left( \left( \mathbf{a}_i \mathbf{a}_i^\top \right)^+ \mathbf{a}_i \mathbf{a}_i^\top \right) = 1.
\end{displaymath}
Also $\tau_i \left( \mathbf{A} \right) \geq 0$ as $\left( \mathbf{A}^\top \mathbf{A} \right)^+ \succeq \mathbf{O}$. 
Hence \textbf{(i)} holds.

\textbf{(ii)} 
It follows that
\begin{displaymath}
    \sum_{i=1}^n \tau_i \left( \mathbf{A} \right) = \mathrm{tr} \left( \mathbf{A} \left( \mathbf{A}^\top \mathbf{A} \right)^+ \mathbf{A}^\top \right) = \mathrm{tr} \left( \left( \mathbf{A}^\top \mathbf{A} \right)^+ \mathbf{A}^\top \mathbf{A} \right) = \mathrm{rank} \left( \mathbf{A} \right).
\end{displaymath}
\end{proof}
The leverage score indicates how important the corresponding row is.
In fact, the following theorem asserts that we can construct a spectral approximation with small approximation size via simple random sampling based on the leverage score.
More precisely, if we are given a \emph{leverage score overestimate}, that is, a vector $\mathbf{u}\in\mathbb{R}^n$ such that, for all $i$, $\tau_i(\mathbf{A}) \leq u_i$, then the number of rows in $\left( 1\pm 1 / \sqrt{\theta} \right)$-spectral approximation is bounded by $\mathrm{O}\left( \theta \| \mathbf{u} \|_1 \log d\right)$.

\begin{theorem} [Spectral Approximation via Row Sampling \cite{DBLP:conf/innovations/CohenLMMPS15}]
\label{theorem:row_sampling}
  Let $\theta$ be an error parameter, and $c$ be a fixed positive constant.
  Let $\mathbf{A} \in \mathbb{R}^{n \times d}$ be a matrix, and $\mathbf{u} \in \mathbb{R}^n$ be a leverage score overestimate.
  We define a sampling probability $p_i = \min \left( \theta \cdot u_i c \log d, 1  \right)$~\textup{($i=1,\dots, n$)}, and construct a random diagonal matrix $\mathbf{S}$ whose $i$-th diagonal element is  $1/\sqrt{p_i} \text{ with probability $p_i$}$ and $0$ otherwise.
  Then, with probability at least $1 - d^{-c/3}$, $\mathbf{SA}$ is a $\left( 1 \pm 1/\sqrt{\theta} \right)$-spectral approximation for $\mathbf{A}$ such that $\mathbf{SA}$ contains at most $\sum_i p_i \le \theta \| \mathbf{u} \|_1 c \log d$ non-zero rows.
\end{theorem}

For example, suppose that we are given a \emph{constant leverage score overestimate}, that is, a vector $\mathbf{u}\in\mathbb{R}^n$ such that, for all $i$, $\tau_i(\mathbf{A}) \leq u_i \le \beta \tau_i(\mathbf{A})$ for some constant $\beta$.
Then the above theorem implies that, by setting $\theta = \epsilon^{-2}$, we can obtain a $(1 \pm \epsilon)$-spectral approximation such that it has $\mathrm{O}\left( r \epsilon ^{-2} \log d \right)$ rows, where $r=\mathrm{rank}(\A)$, as $\sum_i u_i\le \beta \sum_i \tau_i(\mathbf{A})\leq \beta r$ by \cref{lemma:property_leverage_score}.



\subsection{Spectral Approximation in the Online Setting}


In the online setting, a matrix $\mathbf{A}$ is given as a stream of rows, and we receive a row sequentially.
Each time the $i$-th row $\mathbf{a}_i$ arrives, we irrevocably decide whether $\mathbf{a}_i$ is sampled or not.
We cannot access the whole matrix $\A$ in each decision, and thus we cannot compute a standard leverage score.
We introduce a variant of the leverage score, called the \emph{relative leverage score}, that can be computed with the matrix we have at the moment~(corresponding to a matrix $\mathbf{B}$ in the definition below). 
This gives a leverage score overestimate.


\begin{definition} [Relative Leverage Score]
  Let $\mathbf{A}, \mathbf{B} \in \mathbb{R}^{n \times d}$ be matrices.
  For $i\in\{1,2,\dots, n\}$, the \emph{relative leverage score $\tau_i^{\mathbf{B}}(\mathbf{A})$} is defined as follows:
\begin{displaymath}
  \tau_i^{\mathbf{B}}(\mathbf{A}) \overset{\mathrm{def}}{=}
  \mathbf{a}_i^\top \left(
  \begin{pmatrix} \mathbf{B} \\ \mathbf{a}_i^\top \end{pmatrix} ^\top \begin{pmatrix} \mathbf{B} \\ \mathbf{a}_i^\top \end{pmatrix}
  \right)^+ \mathbf{a}_i.
\end{displaymath}
\end{definition}
\noindent
We note that, if $\B$ is a submatrix consisting of rows of $\A$, then $\tau_i^{\mathbf{B}}(\mathbf{A})\geq \tau_i (\mathbf{A})$ holds.


The relative leverage score can be rewritten as follows.
For vectors $\x$ and $\y$, we denote $\x \perp \y$ if $\x^\top \y = 0$.
For a vector $\x$ and a linear subspace $W$, $\x \perp W$ means that $\x \perp \y$ for all $\y \in W$.

\begin{lemma}
\label{lemma:relative_leverage_score_tractable}
For $i\in\{1,2,\dots, n\}$, it holds that
    \begin{displaymath}
        \tau_i^{\mathbf{B}}(\mathbf{A}) = \left\{
        \begin{array}{cl}
          \frac{\mathbf{a}_i^\top \left( \mathbf{B}^\top \mathbf{B} \right)^+ \mathbf{a}_i}{\mathbf{a}_i^\top \left( \mathbf{B}^\top \mathbf{B} \right)^+ \mathbf{a}_i + 1} & \mathrm{if}\  \mathbf{a}_i \perp \mathrm{Ker} (\mathbf{B}), \\
          1 & \mathrm{otherwise.} \\
        \end{array}
    \right.
    \end{displaymath}
\end{lemma}

To prove \cref{lemma:relative_leverage_score_tractable}, we consider the perturbation of Moore-Penrose pseudo-inverses under the rank-1 update operation given in \cite{meyer1973generalized}.
\begin{proposition}[Sherman-Morrison Formula for Moore-Penrose pseudo-inverse \cite{meyer1973generalized}]
\label{prop:sherman-morrison_Moore-Penrose}
Let $\mathbf{A} \in \mathbb{R}^{n \times n}$ be a $\mathrm{PSD}$ matrix, $\mathbf{u} \in \mathbb{R}^n$ be a vector, and $k$ be a real-valued multiplier. If $\mathbf{u} \perp \mathrm{Ker}(\mathbf{A})$, then we have
\begin{displaymath}
    \left( \mathbf{A} + k \mathbf{u}\mathbf{u}^\top \right)^+ = \mathbf{A}^+ - \frac{k \mathbf{A}^+ \mathbf{u} \mathbf{u}^\top \mathbf{A}^+}{1 + k \mathbf{u}^\top \mathbf{A}^+ \mathbf{u}}.
\end{displaymath}
\end{proposition}

\begin{proof}[Proof of \cref{lemma:relative_leverage_score_tractable}]
If $\mathbf{a}_i \perp \mathrm{Ker}( \mathbf{B} )$, by \cref{prop:sherman-morrison_Moore-Penrose}, we obtain
    \begin{align}
    \tau_i^{\mathbf{B}}(\mathbf{A})
    &= \mathbf{a}_i^\top \left( \begin{pmatrix} \mathbf{B} \\ \mathbf{a}_i^\top \end{pmatrix} ^\top \begin{pmatrix} \mathbf{B} \\ \mathbf{a}_i^\top \end{pmatrix} \right)^+ \mathbf{a}_i \nonumber \\
    &= \mathbf{a}_i^\top \left( \left(\mathbf{B}^\top \mathbf{B} \right)^+ - \frac{\left(\mathbf{B}^\top \mathbf{B} \right)^+ \mathbf{a}_i \mathbf{a}_i^\top \left(\mathbf{B}^\top \mathbf{B} \right)^+}{1 + \mathbf{a}_i^\top \left(\mathbf{B}^\top \mathbf{B} \right)^+ \mathbf{a}_i} \right) \mathbf{a}_i \nonumber \\
    &= \frac{\mathbf{a}_i^\top \left( \mathbf{B}^\top \mathbf{B} \right)^+ \mathbf{a}_i}{\mathbf{a}_i^\top \left( \mathbf{B}^\top \mathbf{B} \right)^+ \mathbf{a}_i + 1}.
\end{align}

Next suppose that $\mathbf{a}_i \not\perp \mathrm{Ker}( \mathbf{B} )$.
Then $\dim \left( \mathrm{Im} \left( \mathbf{B}^\top \mathbf{B} + \mathbf{a}_i \mathbf{a}_i^\top \right) \right)$ is exactly one larger than $\dim \left( \mathrm{Im} \left( \mathbf{B}^\top \mathbf{B} \right) \right)$.
Hence there exists a nonzero vector $\mathbf{u} \in \mathrm{Im} \left( \mathbf{B}^\top \mathbf{B} + \mathbf{a}_i \mathbf{a}_i^\top \right)$ such that $\mathbf{u}$ belongs to the orthogonal complement of $\mathrm{Im} \left( \mathbf{B}^\top \mathbf{B} \right)$.
By the definition of the pseudo-inverse, we have
\begin{displaymath}
    \mathbf{u} = \left( \begin{pmatrix} \mathbf{B} \\ \mathbf{a}_i^\top \end{pmatrix} ^\top \begin{pmatrix} \mathbf{B} \\ \mathbf{a}_i^\top \end{pmatrix} \right)^+ \left( \begin{pmatrix} \mathbf{B} \\ \mathbf{a}_i^\top \end{pmatrix} ^\top \begin{pmatrix} \mathbf{B} \\ \mathbf{a}_i^\top \end{pmatrix} \right) \mathbf{u}
    = \left( \begin{pmatrix} \mathbf{B} \\ \mathbf{a}_i^\top \end{pmatrix} ^\top \begin{pmatrix} \mathbf{B} \\ \mathbf{a}_i^\top \end{pmatrix} \right)^+ \mathbf{a}_i \mathbf{a}_i^\top \mathbf{u}.
\end{displaymath}
Multiplying $\mathbf{a}_i^\top$ from the left side, we have
\begin{displaymath}
    \mathbf{a}_i^\top \mathbf{u} = \mathbf{a}_i^\top \left( \begin{pmatrix} \mathbf{B} \\ \mathbf{a}_i^\top \end{pmatrix} ^\top \begin{pmatrix} \mathbf{B} \\ \mathbf{a}_i^\top \end{pmatrix} \right)^+ \mathbf{a}_i \mathbf{a}_i^\top \mathbf{u}
    = \tau_i^{\mathbf{B}} ( \mathbf{A} ) \mathbf{a}_i^\top \mathbf{u}.
\end{displaymath}
As $\mathbf{a}_i^\top \mathbf{u} \neq 0$, we obtain $\tau_i^{\mathbf{B}} ( \mathbf{A} ) = 1$.
\end{proof}

From \cref{lemma:relative_leverage_score_tractable}, $\tau_i^{\mathbf{B}}\left( \mathbf{A} \right)$ can be computed with matrix $\left( \mathbf{B}^\top \mathbf{B} \right)^+$ and $\mathbf{a}_i$. Furthermore, we have $0 \le \tau_i^\mathbf{B} ( \mathbf{A} ) \le 1$, and $\tau_i^\mathbf{B} (\A)$ is equal to $1$ if and only if $\mathbf{a}_i \not \perp \mathrm{Ker}(\mathbf{B})$.
\section{$(1 \pm \epsilon)$-spectral approximation in the online setting}
\label{section:original}

Cohen \emph{et al}.~\cite{DBLP:conf/approx/CohenMP16} presented an algorithm for an $(\epsilon, \delta)$-spectral approximation based on sampling with online ridge leverage scores. In this section, we design an algorithm with relative leverage scores that returns a $(1\pm\epsilon)$-spectral approximation with high probability.
Our algorithm gives a better upper bound on the approximation size.

\subsection{Algorithm}
We describe an algorithm that returns a $(1 \pm \epsilon)$-spectral approximation for a given {\rm PSD} matrix $\A$ as \cref{algorithm:online_row_sampling}. 
In \cref{algorithm:online_row_sampling}, $\tA_i$ is a matrix we have sampled until the end of the $i$-th iteration.
In the $i$-th iteration, we determine a sampling probability $p_i$ with the relative leverage score $\tau_i^{\tA_{i-1}}(\A)$, and append the arriving row $\mathbf{a}_i$ to $\tA_{i-1}$ with probability $p_i$ to obtain $\tA_i$.
This step can be computed with only $\tA_{i-1}$ and $\sa_i$.
In the end, the algorithm returns $\tA_{n}$.

\begin{algorithm}[htbp]
\caption{$\mathsf{Online Row Sampling} \left( \mathbf{A}, \epsilon \right)$}
\label{algorithm:online_row_sampling}
\begin{algorithmic}
\STATE{{\bf Input:} a matrix $\mathbf{A} \in \mathbb{R}^{n \times d}$, an error parameter $\epsilon \in (0,1/2]$.}
\STATE{{\bf Output:} a $(1 \pm \epsilon)$-spectral approximation for $\mathbf{A}$.}
\STATE{Define $c = 3 \epsilon^{-2} \log d $.}
\STATE{$\widetilde{\mathbf{A}}_0 \leftarrow \mathbf{O}$.}
\FOR{$i = 1, \dots, n$}
  \STATE{$\tilde{l}_i \leftarrow \min\left((1+\epsilon) \tau_i^{\widetilde{\mathbf{A}}_{i-1}}\left( \mathbf{A} \right), 1\right)$.}
    \STATE{$p_i \leftarrow \min \left( c\tilde{l}_i, 1 \right)$.}
    \STATE{$\widetilde{\mathbf{A}}_i \leftarrow
    \left\{
        \begin{array}{cl}
            \begin{pmatrix} \widetilde{\mathbf{A}}_{i-1} \\ \mathbf{a}_i^\top /\sqrt{p_i} \end{pmatrix} & \text{with probability $p_i$,} \\
            \widetilde{\mathbf{A}}_{i-1} & \text{otherwise.} \\
        \end{array}
    \right. $}
\ENDFOR
\RETURN $\widetilde{\mathbf{A}}_n.$
\end{algorithmic}
\end{algorithm}

\subsection{Analysis}
\label{subsection:online_row_sampling_analysis}

In this section, we prove that \cref{algorithm:online_row_sampling} satisfies the conditions of \cref{theorem:approximation_size_in_online_row_sampling} with high probability: \cref{algorithm:online_row_sampling} returns a $(1 \pm \epsilon)$-spectral approximation, and its approximation size is $\mathrm{O} \left( \left( r \log \mu (\A) + r + \log d \right) \epsilon^{-2} \log d \right)$.

It turns out that the analysis with a matrix martingale by Cohen \emph{et al}.~\cite{DBLP:conf/approx/CohenMP16} can be adapted to show that $\tA_n$ is a $(1 \pm \epsilon)$-spectral approximation.
We remark that a matrix martingale is also used in the proof of \cref{lemma:analysis_random_order_streams} later, where we describe it in more detail.

\begin{lemma}[cf Lemma 3.3 in \cite{DBLP:conf/approx/CohenMP16}]
\label{lemma:analysis_online_setting}
In \cref{algorithm:online_row_sampling}, with high probability, the following holds:
$\widetilde{\mathbf{l}}=(\widetilde{l}_i)$ is a leverage score overestimate, i.e., 
    $\widetilde{l}_i \geq \tau_i(\mathbf{A})$ 
for $1 \le i \le n$, and $\widetilde{\mathbf{A}}_n$ is a $(1 \pm \epsilon)$-spectral approximation for $\mathbf{A}$.
\end{lemma}

The above lemma does not guarantee that $\tA_i$ is a $(1 \pm \epsilon)$-spectral approximation for \emph{all} $i=1,\dots, n$, with high probability.
However, by taking a union bound, it holds that $\mathrm{O}(d)$ of $\tA_i$'s are  $(1 \pm \epsilon)$-spectral approximations with high probability.
This fact, which is summarized as below, will be used to bound the approximation size later.
\begin{corollary}
\label{corollary:analysis_online_setting}
    In \cref{algorithm:online_row_sampling}, it holds with high probability that
    \begin{equation}\label{equation:analysis_online_setting_approximation}
        (1-\epsilon) \mathbf{A}_i^\top \mathbf{A}_i \preceq \widetilde{\mathbf{A}}_i^\top \widetilde{\mathbf{A}}_i \preceq (1+\epsilon) \mathbf{A}_i^\top \mathbf{A}_i
    \end{equation}
for all $i$ such that $\mathbf{a}_{i+1} \not \perp \mathrm{Ker} \left( \widetilde{\mathbf{A}}_i \right)$.
\end{corollary}
\begin{proof}
    The number of $i$ such that $\mathbf{a}_i \not \perp \mathrm{Ker} \left( \mathbf{A}_{i-1} \right)$ is equal to $\mathrm{rank}(\mathbf{A}) = \mathrm{O} (d)$. 
    It follows from \cref{lemma:analysis_online_setting} that, for a given positive integer $i$, the probability that the relation \eqref{equation:analysis_online_setting_approximation} holds is at least $1-1/d^c$ for a fixed positive constant $c \ge 2$.
    By taking a union bound, the probability that  the relation \eqref{equation:analysis_online_setting_approximation} holds for all $i$ such that $\mathbf{a}_{i+1} \not \perp \mathrm{Ker} \left( \widetilde{\mathbf{A}}_i \right)$ is at least $1-1/d^{c-1}$.
    Thus the desired inequalities hold simultaneously with high probability.
\end{proof}

In what follows, we will show that the approximation size is bounded.
Since $\widetilde{\mathbf{l}}=(\widetilde{l}_i)$ is a leverage score overestimate by \cref{lemma:analysis_online_setting}, it follows from \cref{theorem:row_sampling} that $\widetilde{\mathbf{A}}_n$ has $\mathrm{O}\left( \epsilon^{-2} \| \widetilde{\mathbf{l}} \|_1 \log d\right)$ rows where we set $\theta = \epsilon^{-2}$. 
Hence it suffices to bound $ \| \widetilde{\mathbf{l}} \|_1 = \sum_{i=1}^n \widetilde{l}_i$.
We evaluate it with the pseudo-determinant.

\begin{definition} [Pseudo-Determinant]
\label{definition:pseudo-determinant}
  Let $\mathbf{A} \in \mathbb{R}^{n \times n}$ be a square matrix. The \emph{pseudo-determinant} of $\mathbf{A}$, $\mathrm{Det}(\mathbf{A})$, is defined as the product of its non-zero eigenvalues. Note that $\mathrm{Det}(\mathbf{O})$ is defined as $1$ for convenience.
\end{definition}
The next lemma characterizes the pseudo-determinant, which is often regarded as the definition of the pseudo-determinant.
\begin{lemma}[Characterization of the Pseudo-Determinant \cite{pseudo18}]\label{def:pseudo}
  Let $\mathbf{A} \in \mathbb{R}^{n \times n}$ be a square matrix of rank $r$. 
  Then it holds that
  \begin{displaymath}
    \mathrm{Det}(\mathbf{A}) = \lim_{\delta \rightarrow 0} \frac{\det \left(\mathbf{A} + \delta \mathbf I \right)}{\delta^{n-r}}.
  \end{displaymath}
\end{lemma}
The pseudo-determinant is a similar notion to the determinant, but it does not inherit all the properties of the determinant.
For example, $\mathrm{Det}(\mathbf{A})\mathrm{Det}(\mathbf{B})$ is not always equal to $\mathrm{Det}(\mathbf{AB})$ for square matrices $\A, \mathbf{B}$.
However, we can prove a weaker version of the matrix determinant lemma~(\cref{lemma:det}) for pseudo-determinant as follows.
\begin{lemma}[Matrix Determinant Lemma] \label{lemma:det}
  Let $\mathbf{A} \in \mathbb{R}^{n \times n}$ be a nonsingular matrix, and $\mathbf{u} \in \mathbb{R}^n$ be a vector. Then it holds that
  \begin{displaymath}
    \det \left( \mathbf{A} +\mathbf{u}\mathbf{u}^\top \right) = \det(\mathbf{A}) \left( 1 +\mathbf{u}^\top \mathbf{A}^{-1}\mathbf{u} \right).
  \end{displaymath}
\end{lemma}
\begin{lemma} [Matrix Pseudo-Determinant Lemma]
\label{lemma:Det}
  Let $\mathbf{A} \in \mathbb{R}^{n \times n}$ be a symmetric matrix of rank $r$, and $\mathbf{u} \in \mathbb{R}^n$ be a vector satisfying that $\mathbf{u} \perp \mathrm{Ker}(\mathbf{A})$. Then it holds that
  \begin{equation}\label{eqn:matrix_pseudo_det_lemma}
    \mathrm{Det} \left( \mathbf{A}+\mathbf{u}\mathbf{u}^\top \right) = \mathrm{Det}(\mathbf{A}) \left( 1+\mathbf{u}^\top \mathbf{A}^+\mathbf{u} \right).
  \end{equation}
\end{lemma}
\begin{proof}
  Suppose that the eigenvalues of $\mathbf{A}$ are $\lambda_1, \dots, \lambda_n$
  such that $| \lambda_1 | \geq \dots \geq | \lambda_n | $.
  Then $\lambda_i \neq 0$ for $1 \le i \le r$ and $\lambda_i = 0$ for $r+1 \le i \le n$.
  Moreover, $\mathbf{A}$ is decomposed as $\sum_{i=1}^{n} \lambda_i \mathbf{v}_i \mathbf{v}_i^\top$, where $\{ \mathbf{v}_i \}$ is a family of orthonormal eigenvectors. 
  Set a real value $\delta$ such that $0 < | \delta | < | \lambda_r |$.
  It follows that $\A+\delta \mathbf{I}$ is nonsingular, and $(\A+\delta \mathbf{I})^{-1}=\sum_{i=1}^{n} (\lambda_i + \delta)^{-1} \mathbf{v}_i \mathbf{v}_i^\top$.
  By \cref{lemma:det}, 
  \begin{align}
    \det \left( \mathbf{A} + \delta \mathbf{I} +\mathbf{u}\mathbf{u}^\top \right)
    &= \det \left( \mathbf{A} + \delta \mathbf{I} \right) \left( 1+\mathbf{u}^\top \left(\mathbf{A} + \delta \mathbf{I} \right) ^{-1}\mathbf{u}\right) \nonumber \\
    &= \det \left( \mathbf{A} + \delta \mathbf{I} \right) \left( 1+\mathbf{u}^\top \left( \sum_{i=1}^{n} (\lambda_i + \delta)^{-1} \mathbf{v}_i \mathbf{v}_i^\top \right)\mathbf{u} \right). \nonumber
  \end{align}
  Since $\mathbf{v}_i \in \mathrm{Ker}(\mathbf{A})$ for $r+1 \le i \le n$, we have $\mathbf{v}_i^\top \mathbf{u}=0$ for $r+1 \le i \le n$.
  This implies that
  \begin{align}
    \det \left( \mathbf{A} + \delta \mathbf{I} +\mathbf{u}\mathbf{u}^\top \right) = \det \left( \mathbf{A} + \delta \mathbf{I} \right) \left( 1+\mathbf{u}^\top \left( \sum_{i=1}^{r} (\lambda_i + \delta)^{-1} \mathbf{v}_i \mathbf{v}_i^\top \right)\mathbf{u} \right). \nonumber
  \end{align}
  In addition, as $\mathbf{u} \perp \mathrm{Ker}(\mathbf{A})$, $\mathrm{rank}\left( \mathbf{A}+\mathbf{u}\mathbf{u}^\top \right) = r$ holds. 
  Therefore, we obtain
  \begin{displaymath}
      \frac{\det \left(\mathbf{A} +\mathbf{u}\mathbf{u}^\top + \delta \mathbf{I} \right) }{\delta^{n-\mathrm{rank}\left( \mathbf{A}+\mathbf{u}\mathbf{u}^\top \right)}}
    = \frac{\det \left( \mathbf{A} + \delta \mathbf{I} \right)}{\delta^{n-r}} \left( 1+\mathbf{u}^\top \left( \sum_{i=1}^{r} (\lambda_i + \delta)^{-1} \mathbf{v}_i \mathbf{v}_i^\top \right)\mathbf{u} \right).
  \end{displaymath}
  As $\delta$ approaches $0$, we obtain \cref{eqn:matrix_pseudo_det_lemma} by \cref{def:pseudo}.
\end{proof}

When $\mathbf{u} \not\perp \mathrm{Ker} (\mathbf{A})$, we evaluate $\mathrm{Det} \left(\mathbf{A} +\mathbf{u}\mathbf{u}^\top \right)$ in a different way.
The lemma below is proved using a well-known fact~(\cref{lemma:eigenvalue_perturbation}).

\begin{lemma}
\label{lemma:Detequation:ieq}
Let $\mathbf{A} \in \mathbb{R}^{n \times n}$ be a {\rm PSD} matrix of rank $r$, and $\mathbf{u} \in \mathbb{R}^n$ be a vector such that $\mathbf{u} \not \perp \mathrm{Ker}(\mathbf{A})$. 
Then we have
  \begin{displaymath}
      \mathrm{Det} \left(\mathbf{A} + \mathbf{u} \mathbf{u}^\top\right) \geq \lambda_{\min}\left( \mathbf{A} + \mathbf{u} \mathbf{u}^\top \right) \mathrm{Det}(\mathbf{A}),
  \end{displaymath}
where $\lambda_{\min}(\X)$ for a matrix $\X$ is the smallest non-zero eigenvalue. 
\end{lemma}


\begin{lemma}
\label{lemma:eigenvalue_perturbation}
    Let $\mathbf{A}, \B \in \mathbb{R}^{n \times n}$ be symmetric matrices such that $\mathbf{A} \preceq \mathbf{B}$. 
    Define the $i$-th largest eigenvalue of $\mathbf{A}$ as $\lambda_i \in \mathbb{R}$ and the the $i$-th largest eigenvalue of $\B$ as $\mu _i \in \mathbb{R}$. 
    Then we have $\lambda_i \le \mu_i$ for $1 \le i \le n$.
\end{lemma}

\begin{proof}[Proof of \cref{lemma:Detequation:ieq}]
  If $\mathbf{u} \not \perp \mathrm{Ker}(\mathbf{A})$,  then 
  $\mathrm{rank} (\mathbf{A} + \mathbf{u} \mathbf{u}^\top) = r+1$. By \cref{lemma:eigenvalue_perturbation}, the largest $r$ eigenvalues of $\mathbf{A} + \mathbf{u} \mathbf{u}^\top$ are greater than or equal to those of $\mathbf{A}$, respectively, since $\mathbf{A} + \mathbf{u} \mathbf{u}^\top \succeq \A$. The $(r+1)$-st eigenvalue of $\mathbf{A} + \mathbf{u} \mathbf{u}^\top$ is equal to $\lambda_{\min}\left( \mathbf{A} + \mathbf{u} \mathbf{u}^\top \right)$. Thus we obtain $\mathrm{Det} \left(\mathbf{A} + \mathbf{u} \mathbf{u}^\top \right) \geq \lambda_{\min}\left( \mathbf{A} + \mathbf{u} \mathbf{u}^\top \right) \mathrm{Det}(\mathbf{A})$.
\end{proof}

Using \cref{lemma:Det} and \cref{lemma:Detequation:ieq}, we bound the sum of $\widetilde{l}_i$'s.
\begin{lemma}
\label{lemma:sum_leverage_score}
In \cref{algorithm:online_row_sampling}, we have with high probability
\begin{displaymath}
    \sum_{i=1}^n \widetilde{l}_i = \mathrm{O} \left( r \log \mu (\A) + r + \log d \right).
\end{displaymath}
\end{lemma}
\begin{proof}
    In the proof, we assume that $\widetilde{\mathbf{A}}_n$ is a $(1 \pm \epsilon)$-spectral approximation for $\mathbf{A}$, and the relation \eqref{equation:analysis_online_setting_approximation} holds for all $i$ such that $\mathbf{a}_{i+1} \not \perp \mathrm{Ker} \left( \widetilde{\mathbf{A}}_i \right)$.
    They follow with high probability from \cref{lemma:analysis_online_setting} and \cref{corollary:analysis_online_setting}.
    
    Define $\delta_i$ as
    \begin{displaymath}
        \delta_i \overset{\mathrm{def}}{=} \log \left( \mathrm{Det} \left( \widetilde{\mathbf{A}}_i^\top \widetilde{\mathbf{A}}_i \right) \right) - \log \left( \mathrm{Det} \left( \widetilde{\mathbf{A}}_{i-1}^\top \widetilde{\mathbf{A}}_{i-1} \right) \right)
    \end{displaymath}
    for $1\le i\le n$.
    We observe that 
    \[
    \sum_{i=1}^n \delta_i = \log \left( \mathrm{Det} \left( \widetilde{\mathbf{A}}_n^\top \widetilde{\mathbf{A}}_n \right) \right).
    \]

    Set $\xi = \log \left( \underset{1 \le i \le n}{\min}\  \lambda_{\min} \left( \mathbf{A}_i^\top \mathbf{A}_i \right) \right)$.
    We will first show that, for any $k\geq 0$, 
    \begin{align}
        \label{equation:diff}
        \mathbb{E} \left[ \exp \left( \sum_{i=1}^k \left( \frac{\widetilde{l}_i}{8} - \delta_i \right) \right) \right] \leq \exp \left( \mathrm{rank} \left( \widetilde{\mathbf{A}}_k \right) \left( 1 - \xi \right) \right),
    \end{align}
    which means that the difference between $\sum_{i=1}^k \widetilde{l}_i/8$ and $\log \left( \mathrm{Det} \left( \widetilde{\mathbf{A}}_k^\top \widetilde{\mathbf{A}}_k \right) \right)$ is at most $\mathrm{rank} \left( \widetilde{\mathbf{A}}_k \right) \left( 1 - \xi \right)$ in expectation.
    Since the inequality \eqref{equation:diff} holds trivially when $k = 0$, we suppose that $k\geq 1$ and the inequality \eqref{equation:diff} holds for $k-1$.
    We consider two cases separately: the case when $\mathbf{a}_k \perp \mathrm{Ker} \left( \widetilde{\mathbf{A}}_{k-1} \right)$ and the other case, conditioned on that we know $\tA_{k-1}$.
    
    Suppose that $\mathbf{a}_k \perp \mathrm{Ker} \left( \widetilde{\mathbf{A}}_{k-1} \right)$.
    Then it holds by \cref{lemma:Det} that
    \[
    \delta_k=
    \begin{cases}
    \log \left(1 + \frac{\mathbf{a}_k^\top \left( \widetilde{\mathbf{A}}_{k-1}^\top \widetilde{\mathbf{A}}_{k-1} \right)^+ \mathbf{a}_k}{p_k} \right) & \text{with probability } p_k\\
    0 & \text{with probability } 1-p_k.
    \end{cases}
    \]
    Hence, since $\tilde{l}_k\leq 1$, we see 
    \begin{align}
        \mathbb{E} \left[ \left. \exp \left( \frac{\widetilde{l}_k}{8} - \delta_k \right) \right| \widetilde{\mathbf{A}}_{k-1} \right] 
        &= \exp \left( \frac{\widetilde{l}_k}{8} \right) \left( p_k \cdot \left(1 + \frac{\mathbf{a}_k^\top \left( \widetilde{\mathbf{A}}_{k-1}^\top \widetilde{\mathbf{A}}_{k-1} \right)^+ \mathbf{a}_k}{p_k} \right)^{-1} + \left(1 - p_k \right) \cdot 1 \right) \nonumber \\
        &\le \left( 1 + \frac{\widetilde{l}_k}{4} \right) \left( p_k \cdot \left(1 + \frac{\mathbf{a}_k^\top \left( \widetilde{\mathbf{A}}_{k-1}^\top \widetilde{\mathbf{A}}_{k-1} \right)^+ \mathbf{a}_k}{p_k} \right)^{-1} + \left(1 - p_k \right) \cdot 1 \right) \nonumber\\
        &\le \left( 1 + \frac{\widetilde{l}_k}{4} \right) \left( p_k \cdot \left( 1 + \frac{\tau_k^{\widetilde{\mathbf{A}}_{k-1}}( \mathbf{A} )}{p_k} \right)^{-1} + \left(1 - p_k \right) \cdot 1 \right),
        \label{equation:siki1}
    \end{align}
    where the last inequality follows since $\mathbf{a}_k^\top \left( \widetilde{\mathbf{A}}_{k-1}^\top \widetilde{\mathbf{A}}_{k-1} \right)^+ \mathbf{a}_k\geq \tau_k^{\widetilde{\mathbf{A}}_{k-1}}( \mathbf{A} )$ by \cref{lemma:pseudoinverse}.
    If $p_k < 1$, then $p_k= c \widetilde{l}_k$, and hence $\widetilde{l}_k = (1 + \epsilon) \tau_k^{\widetilde{\mathbf{A}}_{k-1}}(\A)$.
    Hence we obtain
    \begin{displaymath}
         \left( 1 + \frac{\tau_k^{\widetilde{\mathbf{A}}_{k-1}}( \mathbf{A} )}{p_k} \right)^{-1}
        = \left(1 + \frac{1}{c(1+\epsilon)} \right)^{-1}
        \le 1 - \frac{1}{4c}
    \end{displaymath}
    as $\epsilon\le 1/2$. 
    Substituting it into the inequality \eqref{equation:siki1}, we have
    \begin{align*}
        \mathbb{E} \left[ \left. \exp \left( \frac{\widetilde{l}_k}{8} - \delta_k \right) \right| \tA_{k-1} \right]
        &\le \left( 1 + \frac{\widetilde{l}_k}{4} \right) \left( c \widetilde{l}_k \cdot \left(1 - \frac{1}{4c} \right) + \left( 1 - c \widetilde{l}_k \right) \cdot 1 \right) \nonumber \\
        &= \left( 1 + \frac{\widetilde{l}_k}{4} \right) \left( 1 - \frac{\widetilde{l}_k}{4} \right) \le 1. 
    \end{align*}
    On the other hand, if $p_k = 1$, since $\widetilde{l}_k\le (1 + \epsilon) \tau_k^{\widetilde{\mathbf{A}}_{k-1}}( \mathbf{A} ) \le 2\tau_k^{\widetilde{\mathbf{A}}_{k-1}}( \mathbf{A} )$, it follows from the inequality \eqref{equation:siki1} that 
    \begin{align*}
        \mathbb{E} \left[ \left. \exp \left( \frac{\widetilde{l}_k}{8} - \delta_k \right) \right| \tA_{k-1} \right]
        &\le \left( 1 + \frac{\widetilde{l}_k}{4} \right) \left(1 + \tau_k^{\widetilde{\mathbf{A}}_{k-1}}( \mathbf{A} ) \right)^{-1} \nonumber \\
        &\le  \left( 1 + \frac{\widetilde{l}_k}{4} \right) \left(1 + \frac{\widetilde{l}_k}{2} \right)^{-1} \le 1. 
    \end{align*}
    In summary, we have $\mathbb{E} \left[ \left . \exp \left( \widetilde{l}_k / 8 - \delta_k \right) \right| \tA_{k-1} \right] \le 1$.
    Therefore, 
    we obtain the following inequality:
    \begin{align}
        \mathbb{E} \left[ \exp \left( \sum_{i=1}^{k} \left( \frac{\widetilde{l}_i}{8} - \delta_i \right) \right) \right]
        &
        = \mathbb{E} \left[ \mathbb{E} \left[ \left. \exp \left( \frac{\widetilde{l}_k}{8} - \delta_k \right) \right| \tA_{k-1} \right] \exp \left( \sum_{i=1}^{k-1} \left( \frac{\widetilde{l}_i}{8} - \delta_i \right) \right) \right] \nonumber \\
        &\le \mathbb{E} \left[ \exp \left( \sum_{i=1}^{k-1} \left( \frac{\widetilde{l}_i}{8} - \delta_i \right) \right) \right]
        \le \exp \left( \mathrm{rank} \left( \widetilde{\mathbf{A}}_{k-1} \right) \left( 1 - \xi \right) \right), \label{equation:siki4}
    \end{align}
    where the last inequality follows from the induction hypothesis.
    Since $\mathrm{rank} \left( \widetilde{\mathbf{A}}_{k} \right)=\mathrm{rank} \left( \widetilde{\mathbf{A}}_{k-1} \right)$ as $\mathbf{a}_k \perp \mathrm{Ker} \left( \widetilde{\mathbf{A}}_{k-1} \right)$, the inequality \eqref{equation:diff} holds for $k$.
    
    Next suppose that $\mathbf{a}_k \not \perp \mathrm{Ker} \left( \widetilde{\mathbf{A}}_{k-1} \right)$.
    In this case, $\widetilde{l}_k = 1$ and $p_k = 1$. 
    Hence, by \cref{lemma:Detequation:ieq}, we have
    \[
    \delta_k \geq \log \lambda_{\min} \left( \widetilde{\mathbf{A}}_{k-1}^\top \widetilde{\mathbf{A}}_{k-1} + \mathbf{a}_k \mathbf{a}^\top_k\right).
    \]
    In addition, since we assume the relation \eqref{equation:analysis_online_setting_approximation} for $k$, we obtain $(1 - \epsilon) \mathbf{A}_k^\top \mathbf{A}_k \preceq \widetilde{\mathbf{A}}_{k-1}^\top \widetilde{\mathbf{A}}_{k-1} + \mathbf{a}_k \mathbf{a}_k^\top$. 
    This implies that $e^{\delta_k}\geq \lambda_{\min} \left( \widetilde{\mathbf{A}}_{k-1}^\top \widetilde{\mathbf{A}}_{k-1} + \mathbf{a}_k \mathbf{a}^\top_k \right)\geq (1 - \epsilon) \lambda_{\min} \left( \mathbf{A}_k^\top \mathbf{A}_k \right)$.
    Hence we obtain
    \begin{align}
        \mathbb{E} \left[ \left. \exp \left( \frac{\widetilde{l}_k}{8} - \delta_k \right) \right| \widetilde{\mathbf{A}}_{k-1} \right] 
        &\le \left( 1 + \frac{\widetilde{l}_k}{4} \right) e^{-\delta_k}\nonumber \\
        &\le \frac{5}{4} \cdot \frac{1}{(1 - \epsilon) \lambda_{\min} \left( \mathbf{A}_k^\top \mathbf{A}_k \right)} \nonumber \\
        &\le \exp \left( 1 - \xi \right). \nonumber
    \end{align}
    Similarly to the previous case \eqref{equation:siki4}, we have by the induction hypothesis 
    \begin{align*}
        \mathbb{E} \left[ \exp \left( \sum_{i=1}^{k} \left( \frac{\widetilde{l}_i}{8} - \delta_i \right) \right) \right]
        &\le \exp \left( 1 - \xi \right) \mathbb{E} \left[ \exp \left( \sum_{i=1}^{k-1} \left( \frac{\widetilde{l}_i}{8} - \delta_i \right) \right) \right]\nonumber \\
        & \le \exp \left(\left( \mathrm{rank} \left( \widetilde{\mathbf{A}}_{k-1} \right) +1 \right) \left( 1 - \xi \right) \right). 
    \end{align*}
    Since $\mathrm{rank} \left( \widetilde{\mathbf{A}}_{k} \right)=\mathrm{rank} \left( \widetilde{\mathbf{A}}_{k-1} \right)+1$ as $\mathbf{a}_k \not \perp \mathrm{Ker} \left( \widetilde{\mathbf{A}}_{k-1} \right)$, the inequality \eqref{equation:diff} holds for $k$.
    
    Therefore, the inequality \eqref{equation:diff} holds under the assumption that $\widetilde{\mathbf{A}}_n$ is a $(1 \pm \epsilon)$-spectral approximation for $\mathbf{A}$ and the relation \eqref{equation:analysis_online_setting_approximation} holds for all $i$ such that $\mathbf{a}_{i+1} \not \perp \mathrm{Ker} \left( \widetilde{\mathbf{A}}_i \right)$.
    Using Markov's inequality to the inequality \eqref{equation:diff} when $k=n$, we have
    \begin{align*}
        \mathbb{P} \left[ \sum_{i=1}^{n} \left( \frac{\widetilde{l}_i}{8} - \delta_i \right) > r(1 -  \xi) + \log d \right]
        & = 
        \mathbb{P} \left[ \exp \left( \sum_{i=1}^{n} \frac{\widetilde{l}_i}{8} - \delta_i \right) > \exp \left( r(1 -  \xi) + \log d \right) \right] \\
        &\leq \frac{ \exp \left( r \left( 1 - \xi \right) \right) }{ \exp \left( r(1 -  \xi) + \log d \right)}
        \leq \frac{1}{d}.
    \end{align*}
    Therefore, $\sum_{i=1}^{n} \left( \frac{\widetilde{l}_i}{8} - \delta_i \right) \leq r(1 -  \xi) + \log d$ is satisfied with high probability.

    Since $\mathrm{Det} \left( \widetilde{\mathbf{A}}_n^\top \widetilde{\mathbf{A}}_n \right) \le \left( (1+\epsilon) \left( \| \mathbf{A} \|_2^2 \right) \right)^r$, we have $\sum_i \delta_i = \log \mathrm{Det} \left( \widetilde{\mathbf{A}}_n^\top \widetilde{\mathbf{A}}_n \right) \le r +  r \log \| \mathbf{A} \|_2^2$.
    Hence, with high probability, it holds that
    \begin{align}
        \sum_{i=1}^n \widetilde{l}_i &\le 8 \left( \sum_{i=1}^n \delta_i + r(1 - \xi) + \log d \right) \nonumber \\
        &\le 16 r + 8r \log \frac{\| \mathbf{A} \|_2^2}{\underset{1 \le i \le n}{\min}\  \lambda_{\min} \left( \mathbf{A}_i^\top \mathbf{A}_i \right)} + \log d = \mathrm{O} \left( r \log \mu (\A) + r + \log d \right).\nonumber 
    \end{align}
\end{proof}


\begin{proof}[Proof of \cref{theorem:approximation_size_in_online_row_sampling}]
    By \cref{lemma:analysis_online_setting}, \cref{algorithm:online_row_sampling} returns a $(1 \pm \epsilon)$-spectral approximation with high probability.
    Moreover, since $\widetilde{\mathbf{l}}=(\widetilde{l}_i)$ is a leverage score overestimate, 
    it follows from \cref{theorem:row_sampling} that $\widetilde{\mathbf{A}}_n$ has $\mathrm{O}\left(\epsilon^{-2} \| \widetilde{\mathbf{l}} \|_1  \log d\right)$ non-zero rows where we set $\theta = \epsilon^{-2}$. 
    Since \cref{lemma:sum_leverage_score} implies that $\| \widetilde{\mathbf{l}} \|_1 = \mathrm{O}\left(r \log \mu (\A) + r + \log d \right)$, the approximation size is with high probability $\mathrm{O}\left( \left( r \log \mu (\A) + r + \log d \right) \epsilon^{-2} \log d \right)$.
\end{proof}

For the case when an input matrix is the incidence matrix of a graph, 
the upper bound in \cref{theorem:approximation_size_in_online_row_sampling} can be simplified as follows.

\begin{corollary}
Let $G$ be a simple, edge-weighted graph whose largest and smallest weights are $w_\mathrm{max}$ and $w_\mathrm{min}$, respectively.
Let $\mathbf{A}$ be its incidence matrix of rank $r$.
    Then \cref{algorithm:online_row_sampling} returns, with high probability, a $(1 \pm \epsilon)$-spectral approximation for $\mathbf{A}$ with $\mathrm{O}\left( \left( r \log \left(w_\mathrm{max} / w_\mathrm{min}\right) + r + \log d \right) \right.$ $\left. \epsilon^{-2} \log d \right)$ edges.
\end{corollary}
\begin{proof}
    The largest eigenvalue of the Laplacian matrix is $\mathrm{O}(d \omega_{\mathrm{max}})$. It is known that the smallest non-zero eigenvalue of the Laplacian matrix is $\mathrm{\Omega}(\omega_{\mathrm{min}}/ d^2)$ (from \cite{DBLP:journals/siammax/SpielmanT14}, Lemma 6.1). 
    By applying these facts to \cref{theorem:approximation_size_in_online_row_sampling}, we obtain $m = \mathrm{O}\left( \left( r \log \left(w_\mathrm{max} / w_\mathrm{min}\right) + r + \log d \right) \epsilon^{-2} \log d \right)$.
\end{proof}

Regarding a lower bound on the approximation size in the online setting, we show that in order to construct a $(1 \pm \epsilon)$-spectral approximation $\Omega \left(r \epsilon^{-2} \log \mu (\A) + r\epsilon^{-2} \right)$ rows have to be sampled in the worst case.
This can be shown in the same way as Theorem 5.1 in \cite{DBLP:conf/approx/CohenMP16}. We omit the proof.

\begin{theorem}\label{thm:onlein_lb}
  Let $\epsilon \in (0, 1)$ be an error parameter. Let $R$ be an algorithm that samples rows in the online setting and returns a $(1 \pm \epsilon)$-spectral approximation with probability at least 1/2.
  Then there exists  a matrix $\mathbf{A}$ of rank $r$ in $\mathbb{R}^{n \times d}$ such that $R$ samples $\Omega \left( r \epsilon^{-2} \log \mu (\A) + r\epsilon^{-2} \right)$ rows in expectation.
\end{theorem}

\section{Fast $(1 \pm \epsilon)$-approximation in the online random order setting}
\label{section:main}

In Sections~\ref{section:main} and \ref{section:optimality}, we focus on the online random order setting.
Recall that, in the online random order setting, 
we are given a family of row vectors $\mathbf{X} = \{ \mathbf{x}^\top_1, \mathbf{x}^\top_2, \dots, \mathbf{x}^\top_n \}$ in $\mathbb{R}^d$.
An input matrix $\A$ is chosen from a discrete uniform distribution $\mathcal{A}(\mathbf{X})$ whose element is a matrix obtained by permuting row vectors of $\mathbf{X}$, and is given as a stream of rows in the online setting.

In Section \ref{subsection:scaled_sampling}, we present a simpler algorithm~(\cref{algorithm:scaled_sampling}) such that it returns a $(1 \pm \epsilon)$-spectral approximation with approximation size $\mathrm{O}(d\epsilon^{-2}\log n\log d)$, but it runs in less efficient time and space complexity.
In \cref{subsection:less_memory}, with further observations, we develop \cref{algorithm:scaled_sampling} into a less running time and space algorithm, which gives \cref{theorem:improved_scaled_sampling}.
Moreover, in \cref{subsection:less_runtime_memory}, we reduce the space complexity with a semi-streaming algorithm.

\subsection{Simple Scaled Sampling Algorithm}
\label{subsection:scaled_sampling}
We define some notations used in \cref{algorithm:scaled_sampling}. Define $K \overset{\mathrm{def}}{=} d \log d$. For convenience, we may assume that there exists $\alpha \in \mathbb{N}$ such that $n = (2^{\alpha + 1} - 1)K$. 
Thus $\alpha = \log_2 (n / K + 1) - 1$.
We also assume that $c_1 \exp \left( d \right) > n$ for some fixed positive constant $c_1$. 

In \cref{algorithm:scaled_sampling}, we divide $\A$ into blocks.
For $i\in\{0,1,\dots, \alpha\}$, the \emph{$i$-th block} is a matrix composed of $2^i K$ consecutive rows from $\mathbf{a}_{(2^i-1)K+1}$ to $\mathbf{a}_{(2^{i+1}-1)K}$.
Moreover, we denote the matrix consisting of the $0,\dots,i$-th blocks by $\mathbf{M}_i$. 
Similarly to \cref{algorithm:online_row_sampling}, we denote by $\tA_j$ a matrix we have sampled until the $j$-th row arrives, and the output is $\tA_n$.
Additionally, in the $i$-th block,
we keep a $(1 \pm \epsilon)$-spectral approximation for $\mathbf{M}_{i-1}$, denoted by $\widetilde{\mathbf{M}}_{i-1}$.
In the $i$-th block, we compute $\tau_j^{\widetilde{\mathbf{M}}_{i-1}}(\A)$ for a row $j$, and sample the $j$-th row based on it.
Thus we do not need to compute the Moore-Penrose pseudo-inverse each time, but need the Moore-Penrose pseudo-inverse of only one matrix $\widetilde{\mathbf{M}}^\top_{i-1}\widetilde{\mathbf{M}}_{i-1}$ for the $i$-th block.

\begin{algorithm}[htbp]
\caption{$\mathsf{Scaled Sampling}\left(\mathbf{A}, \epsilon \right)$}
\label{algorithm:scaled_sampling}
\begin{algorithmic}[1]
\STATE{{\bf Input:} a matrix $\mathbf{A} \in \mathbb{R}^{n \times d}$, an error parameter $\epsilon \in (0,1/2]$.}
\STATE{{\bf Output:} a $(1 \pm \epsilon)$-spectral approximation for $\mathbf{A}$.}
\STATE{Define $K = d \log d$, $c = 6 \epsilon^{-2} \log d$ and $\alpha = \log_2 (n / K + 1) - 1$.}
\STATE{$\widetilde{\mathbf{A}}_0 \leftarrow \mathbf{O}$.}
\FOR{$j = 1, \dots, K$}
    \STATE{$\widetilde{\mathbf{A}}_j \leftarrow \begin{pmatrix} \widetilde{\mathbf{A}}_{j-1} \\ \mathbf{a}_j^\top \end{pmatrix}$.}
\ENDFOR
\FOR{$i = 1, \dots, \alpha $}
  \STATE{$\widetilde{\mathbf{M}}_{i-1} \leftarrow \widetilde{\mathbf{A}}_{\left(2^i -1 \right) K}$.}
  \FOR{$j = \left( 2^i -1 \right) K + 1, \dots, \left(2^{i+1} -1 \right) K$}
    \STATE{$\tilde{l}_j \leftarrow \min\left((1+\epsilon) \tau_j^{\widetilde{\mathbf{M}}_{i-1}}\left( \mathbf{A} \right), 1\right)$.}
    \STATE{$p_j \leftarrow \min \left( c\tilde{l}_j, 1 \right)$.}
    \STATE{$\widetilde{\mathbf{A}}_j \leftarrow
    \left\{
        \begin{array}{cl}
            \begin{pmatrix} \widetilde{\mathbf{A}}_{j-1} \\ \mathbf{a}_j^\top /\sqrt{p_j} \end{pmatrix} & \text{with probability $p_j$,} \\
            \widetilde{\mathbf{A}}_{j-1} & \text{otherwise.} \\
        \end{array}
    \right. $}
  \ENDFOR
\ENDFOR
\RETURN $\widetilde{\mathbf{A}}_n.$
\end{algorithmic}
\end{algorithm}




\cref{algorithm:scaled_sampling} satisfies the following.

\begin{theorem}
\label{theorem:scaled_sampling}
    Let $\epsilon \in (0, 1/2]$ be an error parameter, and $\X$ be a family of $n$ row vectors in $\mathbb{R}^d$.
    Then $\A \sim \mathcal{A}(\X)$ satisfies the following with high probability:
    \cref{algorithm:scaled_sampling} returns a $(1 \pm \epsilon)$-spectral approximation for $\mathbf{A}$ with $\mathrm{O}\left( d \epsilon^{-2} \log n \log d \right)$ rows with high probability.
    It consumes $\mathrm{O}\left( \mathrm{nnz}(\mathbf{A}) \log n + \left( d^\omega + d^2 \log n \right) \epsilon^{-2} \log^2 n \log d \right)$ time and stores $\mathrm{O}\left( d \epsilon^{-2} \log n \log d \right)$ rows.
\end{theorem}

In what follows, we prove that \cref{algorithm:scaled_sampling} satisfies \cref{theorem:scaled_sampling}, that is, it returns a $(1 \pm \epsilon)$-spectral approximation for $\mathbf{A}_i$, the approximation size is small, and the algorithm runs in desired time and space, respectively.

\subsubsection{Spectral Approximation}

We prove that, with high probability, $\widetilde{\mathbf{l}}= (\widetilde{l}_i)$ is a leverage score overestimate and the algorithm returns a $(1 \pm \epsilon)$-spectral approximation for $\mathbf{A}$.
This holds independently of row permutations, that is, this holds for any $\A\sim \mathcal{A}(\X)$.
The proof uses a matrix martingale similarly to proving \cref{lemma:analysis_online_setting} and Lemma 3.3 in \cite{DBLP:conf/approx/CohenMP16}.

\begin{definition} [Matrix Martingale]
\label{definition:matrix_martingale}
     Let $\{ \mathbf{Y}_j : j = 0,1,2,\dots, n\}$ be a discrete-time stochastic process whose values are matrices of finite dimension. Then $\{\mathbf{Y}_j\}$ is a \emph{matrix martingale} if the following two conditions hold:
    \begin{itemize}
        \item[\rm (i)] $\mathbb{E} \left[ \| \mathbf{Y}_j \|_2 \right] < +\infty$,
        \item[\rm (ii)] $\mathbb{E}_{j-1} \left[ \mathbf{Y}_j \right] = \mathbf{Y}_j$.
    \end{itemize}
\end{definition}

We introduce an exponential tail bound for a matrix martingale which ensures that the final state is well-bounded under small variation.

\begin{theorem} [Freedman's Inequality for Matrix Martingale \cite{tropp11}] \label{theorem:martingale}
  Let $\{ \mathbf{Y}_j : j = 0,1,2,\dots, n\}$ be a matrix martingale with the sequence of self-adjoint $d \times d$ matrices.
  Define $\D_j \overset{\mathrm{def}}{=} \Y_j - \Y_{j-1}$ for $j=1,\dots, n$ as the difference sequence of $\{\Y_j\}$. 
  Assume that $\{\D_j\}$ is uniformly bounded as 
    $\| \mathbf{D}_j \|_2 \leq R \ \mathrm{a.s.} \ \mathrm{for}\ j=1,2,\dots,n$.
  We define $\mathbf{W}_j$ as the predictable variation process of the martingale:
  \begin{displaymath}
    \mathbf{W}_j \overset{\mathrm{def}}{=} \sum_{\ell=1}^{j} \mathbb{E}_{\ell-1} \left[ \mathbf{D}_\ell^2 \right] \  \mathrm{for} \ j=1,2,\dots,n.
  \end{displaymath}
  Then, for all $\epsilon \geq 0$ and $\sigma^2 > 0$, it holds that
  \begin{displaymath}
    \mathbb{P} \left[ \| \mathbf{Y}_n \|_2 \geq \epsilon \ and \ \| \mathbf{W}_n \|_2 \leq \sigma^2 \right] \leq d \cdot \exp \left( {- \frac{\epsilon^2/2}{\sigma^2 + R\epsilon/3}} \right).
  \end{displaymath}
\end{theorem}
\begin{lemma}
\label{lemma:analysis_random_order_streams}
In \cref{algorithm:scaled_sampling}, for any matrix $\A \in \mathbb{R}^{n \times d}$ the following holds with high probability:
$\widetilde{\mathbf{l}}=(\widetilde{l}_j)$ is a leverage score overestimate, i.e., 
        $\widetilde{l}_j \geq \tau_j(\mathbf{A})$ 
    for $1 \le j \le n$, and $\widetilde{\mathbf{A}}_n$ is a $(1 \pm \epsilon)$-spectral approximation for $\mathbf{A}$.
\end{lemma}
\begin{proof}
  Let 
  \begin{displaymath}
    \mathbf{u}_j \overset{\mathrm{def}}{=} \left( \mathbf{A}^\top \mathbf{A} \right)^{+/2} \mathbf{a}_j
  \end{displaymath}
  for $1 \le j \le n$.
  We define a matrix martingale $\{ \mathbf{Y}_j \}$ recursively as follows.
  Define $\mathbf{Y}_0 = \mathbf{O}$.
  For $j\geq 1$, if $\| \mathbf{Y}_{j-1} \|_2 \geq \epsilon$, set $\mathbf{D}_j = \mathbf{O}$, and otherwise, set $\mathbf{D}_j$ to
  \begin{displaymath}
    \mathbf{D}_j = 
    \left\{
        \begin{array}{cl}
            (1/p_j - 1)\mathbf{u}_j \mathbf{u}_j^\top & \text{with probability $p_j$}, \\
            -\mathbf{u}_j \mathbf{u}_j^\top & \text{otherwise,} \\
        \end{array}
    \right.
  \end{displaymath}
  and then define $\mathbf{Y}_j = \mathbf{Y}_{j-1} + \mathbf{D}_j$.

  In what follows, we suppose that $\| \mathbf{Y}_{j-1} \|_2 < \epsilon$.
  We note by definition that $\| \mathbf{Y}_{j'} \|_2 < \epsilon$ for any $j' < j$. $\mathbf{Y}_{j}$ can be expressed as

  \begin{equation}
    \mathbf{Y}_{j} = \left(\mathbf{A}^\top \mathbf{A} \right)^{+/2} \left( \widetilde{\mathbf{A}}_{j}^\top \widetilde{\mathbf{A}}_{j} - \mathbf{A}_{j}^\top \mathbf{A}_{j} \right) \left( \mathbf{A}^\top \mathbf{A} \right)^{+/2}, \label{equation:martingale}
  \end{equation}
  which follows by induction on $j$.
  Indeed, by induction, we have
  \begin{displaymath}
      \mathbf{Y}_{j} = \left(\mathbf{A}^\top \mathbf{A} \right)^{+/2} \left( \widetilde{\mathbf{A}}_{j-1}^\top \widetilde{\mathbf{A}}_{j-1} - \mathbf{A}_{j-1}^\top \mathbf{A}_{j-1} \right) \left( \mathbf{A}^\top \mathbf{A} \right)^{+/2} + \mathbf{D}_{j}.
  \end{displaymath}
  Since we observe $\mathbf{D}_{j}=\left( \mathbf{A}^\top \mathbf{A} \right)^{+/2}\mathbf{Z}_{j}\left( \mathbf{A}^\top \mathbf{A} \right)^{+/2}$, where 
  \begin{displaymath}
      \mathbf{Z}_{j} = 
        \left\{
            \begin{array}{cl}
                 (1/p_{j} - 1) \mathbf{a}_{j} \mathbf{a}_{j}^\top  & \text{with probability $p_{j}$}, \\
                  - \mathbf{a}_{j} \mathbf{a}_{j}^\top   & \text{otherwise,} \\
            \end{array}
        \right.
  \end{displaymath}
  it holds that
    \begin{displaymath}
      \mathbf{Y}_{j} = \left( \mathbf{A}^\top \mathbf{A} \right)^{+/2}\left(\widetilde{\mathbf{A}}_{j-1}^\top \widetilde{\mathbf{A}}_{j-1} - \mathbf{A}_{j-1}^\top \mathbf{A}_{j-1}+ \mathbf{Z}_{j}\right)\left( \mathbf{A}^\top \mathbf{A} \right)^{+/2}.
  \end{displaymath}
  This is equal to \cref{equation:martingale} by the definition of $\tA_{j}$ in \cref{algorithm:scaled_sampling}, noting that $\mathbf{A}_{j}^\top \mathbf{A}_{j}=\mathbf{A}_{j-1}^\top \mathbf{A}_{j-1}+\mathbf{a}_{j} \mathbf{a}_{j}^\top$.
  Thus \cref{equation:martingale} holds. 
  
  It follows from \cref{equation:martingale} with $j=n$ that
  \begin{align}
    \| \mathbf{Y}_{n} \|_2 < \epsilon
    &\Leftrightarrow - \epsilon \mathbf{I} \prec \left(\mathbf{A}^\top \mathbf{A} \right)^{+/2} \left( \widetilde{\mathbf{A}}_n^\top \widetilde{\mathbf{A}}_n - \mathbf{A}^\top \mathbf{A} \right) \left( \mathbf{A}^\top \mathbf{A} \right)^{+/2} \prec \epsilon \mathbf{I} \nonumber \\
    &\Rightarrow - \epsilon \mathbf{A}^\top \mathbf{A} \preceq \widetilde{\mathbf{A}}_n^\top \widetilde{\mathbf{A}}_n - \mathbf{A}^\top \mathbf{A} \preceq \epsilon \mathbf{A}^\top \mathbf{A} \nonumber \\
    &\Leftrightarrow (1 - \epsilon) \mathbf{A}^\top \mathbf{A} \preceq \widetilde{\mathbf{A}}_n^\top \widetilde{\mathbf{A}}_n \preceq (1 + \epsilon) \mathbf{A}^\top \mathbf{A}.  \label{equation:spectral_approximation}
  \end{align}
  This means that $\tA_n$ is a $(1 \pm \epsilon)$-spectral approximation for $\mathbf{A}$. Later, we will show that the probability that $\| \mathbf{Y}_{n} \|_2 \geq \epsilon$ is small. 

  We next show that $\widetilde{\mathbf{l}}=(\widetilde{l}_j)$ is a leverage score overestimate, i.e., $\widetilde{l}_j \geq \tau_j(\mathbf{A})$  for $1 \le j \le n$.
  Since $\| \mathbf{Y}_{j-1} \|_2 < \epsilon$, we have by \cref{equation:martingale} that
  \[
    \widetilde{\mathbf{A}}_{j-1}^\top \widetilde{\mathbf{A}}_{j-1} - \mathbf{A}_{j-1}^\top \mathbf{A}_{j-1} \preceq \epsilon \mathbf{A}^\top \mathbf{A}.
  \]
  Hence it holds that 
  \[
    \begin{pmatrix} \widetilde{\mathbf{A}}_{j-1} \\ \mathbf{a}_j^\top \end{pmatrix} ^\top
    \begin{pmatrix} \widetilde{\mathbf{A}}_{j-1} \\ \mathbf{a}_j^\top \end{pmatrix}
    \preceq \left( \mathbf{A}_{j-1}^\top \mathbf{A}_{j-1} +\epsilon \mathbf{A}^\top \mathbf{A}\right) + \mathbf{a}_{j} \mathbf{a}_{j}^\top
    \preceq \mathbf{A}_{j}^\top \mathbf{A}_{j} + \epsilon \mathbf{A}^\top \mathbf{A}. 
  \]
  Moreover, letting $i$ satisfy $\left(2^i - 1 \right)K + 1 \leq j \leq \left(2^{i+1} - 1\right)K$,  we see that $\widetilde{\mathbf{M}}_{i-1}^\top\widetilde{\mathbf{M}}_{i-1}\preceq \tA_{j-1}^\top \tA_{j-1}$.
  Using the above inequalities and \cref{lemma:pseudoinverse}, we have
  \begin{align}
    \widetilde{l}_j
    &= \min \left( (1+\epsilon)\mathbf{a}_j^\top \left( 
    \begin{pmatrix} \widetilde{\mathbf{M}}_{i-1} \\ \mathbf{a}_j^\top \end{pmatrix} ^\top
    \begin{pmatrix} \widetilde{\mathbf{M}}_{i-1} \\ \mathbf{a}_j^\top \end{pmatrix}
    \right)^+ \mathbf{a}_j, 1 \right) \nonumber \\
    &\geq \min \left( (1+\epsilon)\mathbf{a}_j^\top \left(
    \begin{pmatrix} \widetilde{\mathbf{A}}_{j-1} \\ \mathbf{a}_j^\top \end{pmatrix} ^\top \begin{pmatrix} \widetilde{\mathbf{A}}_{j-1} \\ \mathbf{a}_j^\top \end{pmatrix}
    \right)^+ \mathbf{a}_j, 1 \right) \nonumber \\
    &\geq \min \left( (1+\epsilon) \mathbf{a}_j^\top \left( \mathbf{A}_j^\top \mathbf{A}_j + \epsilon \mathbf{A}^\top \mathbf{A} \right)^+ \mathbf{a}_j, 1 \right) \nonumber \\
    &\geq \min \left( (1+\epsilon) \mathbf{a}_j^\top \left( (1+\epsilon) \mathbf{A}^\top \mathbf{A} \right)^+ \mathbf{a}_j, 1 \right) \nonumber \\
    &\geq \mathbf{u}_j^\top \mathbf{u}_j = \tau_j (\A). \label{eqn:differencefor4.1.3_1}
  \end{align}
  Therefore, $\widetilde{l}_j \geq \tau_j(\mathbf{A})$ holds.
  
  Finally, we evaluate the probability that $\| \mathbf{Y}_{n} \|_2 \geq \epsilon$ by \cref{theorem:martingale}.
  Let $k\in\{1,\dots, n\}$ be the number such that  $\| \mathbf{Y}_k \|_2 \geq \epsilon$ for the first time.
  By the above discussion \eqref{eqn:differencefor4.1.3_1}, we see that $\widetilde{l}_j\geq \mathbf{u}_j^\top \mathbf{u}_j$ for any $j<k$.
  If $p_j < 1$, then $p_j = c \widetilde{l}_j \geq c \mathbf{u}_j^\top \mathbf{u}_j$ holds, and for such $j<k$ we obtain $\mathbf{u}_j \mathbf{u}_j^\top / p_j \preceq (1 / c) \mathbf{I}$. This implies that  $\| \mathbf{u}_j \mathbf{u}_j^\top / p_j \|_2 \le 1/c$. 
  For $j$ with $p_j = 1$, we have $\mathbf{D}_j = \mathbf{O}$.
  Therefore, the predictable variation $\mathbf{W}_{k} = \sum_{j=1}^k \mathbb{E}_{j-1}[\mathbf{D}_j^2]$ can be calculated as
    \begin{align}
      \mathbf{W}_{k}
      &= \sum_{j: p_j \neq 1} \left( p_j \left( \left( \frac{1}{p_j} - 1 \right) \mathbf{u}_j\mathbf{u}_j^\top \right)^2 + \left( 1 - p_j \right) \left( -\mathbf{u}_j\mathbf{u}_j^\top \right)^2 \right) \nonumber \\
      &= \sum_{j: p_j \neq 1} \left( \frac{1}{p_j} - 1 \right) \left( \mathbf{u}_j\mathbf{u}_j^\top \right)^2 \nonumber \\
      &\preceq \sum_{j: p_j \neq 1} \frac{\mathbf{u}_j \mathbf{u}_j^\top}{c} \nonumber \\
      &\preceq \frac{1}{c} \mathbf{I}. \nonumber
  \end{align}
  Hence $\| \mathbf{W}_k \|_2 \le 1/c$. 
  Moreover, since $\mathbf{D}_{k'}$ for any $k' > k$ is $\mathbf{O}$ by the construction of $\{\Y_j\}$,  we obtain $\| \mathbf{W}_n \|_2 = \| \mathbf{W}_k \|_2 \le 1/c$.
  Thus it always holds that $\| \mathbf{W}_n \|_2 \le 1/c$.
  Therefore, by using \cref{theorem:martingale}, we obtain
  \begin{align}
    \mathbb{P} \left[ \| \mathbf{Y}_n \|_2 \geq \epsilon \right]
    &= \mathbb{P} \left[ \text{ $\| \mathbf{Y}_n \|_2 \geq \epsilon$ and $\| \mathbf{W}_n \|_2 \leq 1/c$ } \right] \nonumber \\
    &\leq d \cdot \exp \left( - \frac{\epsilon^2 / 2}{1 / c + \epsilon / 3c} \right) \nonumber \\
    &\leq d \cdot \exp \left( - \frac{3 c \epsilon^2}{8} \right) \nonumber \\
    &\leq \frac{1}{d}. \label{martingale_bound}
  \end{align}
  Thus the probability that $\| \mathbf{Y}_n \|_2 \geq \epsilon$ is at most $1/d$. 
  If we replace $c$ with $cc'$ for any constant $c'$, the probability that $\| \mathbf{Y}_n \|_2 \geq \epsilon$ becomes at most $1/d^{c'}$.
  Thus $\| \mathbf{Y}_n \|_2 < \epsilon$ holds with high probability. Hence the desired inequalities \eqref{equation:spectral_approximation} and \eqref{eqn:differencefor4.1.3_1} for $1 \le j \le n$ hold with high probability.
\end{proof}

Since $c_1 \exp \left( d \right) > n$ for some fixed positive constant $c_1$, we can see that $\widetilde{\mathbf{M}}_i$ is a $(1 \pm \epsilon)$-spectral approximation for $\mathbf{A}_i$ for all $i$ with $1\le i\le \alpha$ with high probability, similarly to \cref{corollary:analysis_online_setting}.

\begin{corollary}
\label{corollary:analysis_random_order_streams}
    In \cref{algorithm:scaled_sampling}, for any matrix $\A \in \mathbb{R}^{n \times d}$ it holds with high probability that
    \begin{displaymath}
        (1-\epsilon) \mathbf{M}_i^\top \mathbf{M}_i \preceq \widetilde{\mathbf{M}}_i^\top \widetilde{\mathbf{M}}_i \preceq  (1+\epsilon) \mathbf{M}_i^\top \mathbf{M}_i
    \end{displaymath}
    for all $i=0,1, \dots,  \alpha - 1$.
\end{corollary}


\subsubsection{Approximation size}

We next argue that the approximation size is bounded in \cref{algorithm:scaled_sampling}.

In the beginning, we introduce a fascinating theorem which ensures that a good leverage score of the original matrix is computed with a matrix obtained by uniform sampling.
\begin{theorem}[Leverage Score Estimation via Uniform Sampling \cite{DBLP:conf/innovations/CohenLMMPS15}]
\label{theorem:uniform_sampling_leverage_score_overestimate}
Let $\mathbf{A} \in \mathbb{R}^{n \times d}$ be a matrix. Let S denote a uniformly random sample of $m$ rows from $\mathbf{A}$ and let $\mathbf{S} \in \mathbb{R}^{n \times n}$ be its diagonal indicator matrix (i.e. $\mathbf{S}_{ii} = 1$ for $i \in S$, $\mathbf{S}_{ii} = 0$ otherwise). Define $\widehat{\mathbf{A}} \overset{\mathrm{def}}{=} \mathbf{SA}$ and  $\widehat{\boldsymbol{\tau}}$ as
\begin{displaymath}
    \widehat{\tau}_i \overset{\mathrm{def}}{=}
    \left\{
        \begin{array}{cl}
            \min \left( \mathbf{a}_i^\top \left( \widehat{\mathbf{A}}^\top \widehat{\mathbf{A}} \right)^+ \mathbf{a}_i, 1 \right) & \mathrm{if}\ \mathbf{a}_i \perp \mathrm{Ker}(\widehat{\mathbf{A}}), \\
            1 & \mathrm{otherwise.} \\
        \end{array}
    \right.
\end{displaymath}
Then $\widehat{\boldsymbol{\tau}}$ is a leverage score overestimate, and $\| \widehat{\boldsymbol{\tau}} \|_1 = \mathrm{O} \left( n d /m \right)$ holds with high probability.
\end{theorem}
We see that $\widehat{\tau}_i$ is $d / m$ in average. 
This can be applied to the following analysis in the online random order setting.

\begin{lemma}\label{lemma:analysis_random_order_streams2}
    Let $\mathbf{X}$ be a family of $n$ row vectors in $\mathbf{R}^d$.
    A matrix $\A \sim \mathcal{A}(\X)$ satisfies the following with high probability:
    \cref{algorithm:scaled_sampling} returns $\widetilde{\mathbf{A}}_n$ that contains with high probability $\mathrm{O} \left( d \epsilon^{-2} \log n \log d \right)$ rows.
\end{lemma}
\begin{proof}
Since $\widetilde{\mathbf{l}}=(\tilde{l}_j)$ is a leverage score overestimate by \cref{lemma:analysis_random_order_streams}, it follows from \cref{theorem:row_sampling} that $\widetilde{\mathbf{A}}_n$ has $\mathrm{O}\left(\epsilon^{-2} \| \widetilde{\mathbf{l}} \|_1  \log d\right)$ non-zero rows where we set $\theta = \epsilon^{-2}$.

In the rest of the proof, we estimate $\| \widetilde{\mathbf{l}} \|_1$.
We first apply \cref{theorem:uniform_sampling_leverage_score_overestimate}, as, for all $1 \le i \le \alpha$, we can regard $\mathbf{M}_{i-1}$ as a matrix uniformly sampled from $\mathbf{M}_i$. 

Let $i\in\{0,1,\dots, \alpha\}$.
For $1\le j \le \left( 2^{i+1}-1 \right)K$, define $\widehat{\tau}_{i,j}$ to be 
\begin{displaymath}
    \widehat{\tau}_{i,j} \overset{\mathrm{def}}{=}
    \left\{
        \begin{array}{cl}
            \min \left( \mathbf{a}_j^\top \left( \mathbf{M}_{i-1}^\top \mathbf{M}_{i-1} \right)^+ \mathbf{a}_j, 1 \right) & \mathrm{if}\ \mathbf{a}_j \perp \mathrm{Ker} \left( \mathbf{M}_{i-1} \right), \\
            1 & \mathrm{otherwise.} \\
        \end{array}
    \right.
\end{displaymath}
Then $(\widehat{\tau}_{i,j})$ is a leverage score overestimate for $\mathbf{M}_{i}$ by \cref{theorem:uniform_sampling_leverage_score_overestimate}.
Moreover, it follows from \cref{theorem:uniform_sampling_leverage_score_overestimate} that, with high probability, 
\begin{align}
    \sum_{j=\left(2^i-1\right)K+1}^{\left(2^{i+1}-1\right)K} \widehat{\tau}_{i,j} \le
    \sum_{j=1}^{\left(2^{i+1}-1\right)K} \widehat{\tau}_{i,j} =
    \mathrm{O} \left( \frac{\left( 2^{i+1} - 1 \right)K}{\left( 2^i - 1 \right)K} d \right) =
    \mathrm{O} \left( d \right). \label{equation:leverage_score_overestimate}
\end{align}
By taking a union bound, \cref{equation:leverage_score_overestimate} holds for all $0 \le i \le \alpha$ with high probability.
Hence we obtain
\begin{equation}
    \sum_{i=0}^\alpha \sum_{j=\left(2^i-1\right)K+1}^{\left(2^{i+1}-1\right)K} \widehat{\tau}_{i,j} = \mathrm{O} \left(d \log n \right).\label{eq:tau_ij}
\end{equation}

We will show that $\widetilde{l}_j \le C \widehat{\tau}_{i,j}$ holds for some constant $C$ for $1 \le i \le \alpha$ and $\left( 2^i-1 \right) K + 1 \le j \le \left( 2^{i+1}-1 \right)K$, 
which implies that $\sum \widetilde{l}_j  = \mathrm{O} \left(d \log n \right)$ by \cref{eq:tau_ij}. 
If $\mathbf{a}_j \perp \mathrm{Ker} \left( \mathbf{M}_{i-1} \right)$, since we have $\left( \widetilde{\mathbf{M}}_{i-1}^\top \widetilde{\mathbf{M}}_{i-1} \right)^+ \le (1-\epsilon)^{-1} \left( \mathbf{M}_{i-1}^\top \mathbf{M}_{i-1} \right)^+$ from \cref{corollary:analysis_random_order_streams}, 
it holds that
\begin{align}
    \widetilde{l}_j
    &= \min \left( (1+\epsilon)\mathbf{a}_j^\top \left( 
    \begin{pmatrix} \widetilde{\mathbf{M}}_{i-1} \\ \mathbf{a}_j^\top \end{pmatrix} ^\top
    \begin{pmatrix} \widetilde{\mathbf{M}}_{i-1} \\ \mathbf{a}_j^\top \end{pmatrix}
    \right)^+ \mathbf{a}_j, 1 \right) \nonumber \\
    &\leq \min \left( (1+\epsilon)\mathbf{a}_j^\top \left( 
    \widetilde{\mathbf{M}}_{i-1}^\top \widetilde{\mathbf{M}}_{i-1}
    \right)^+ \mathbf{a}_j, 1 \right) \nonumber \\
    &\leq \min \left( \frac{1+\epsilon}{1-\epsilon}\mathbf{a}_j^\top \left( 
    \mathbf{M}_{i-1}^\top \mathbf{M}_{i-1}
    \right)^+ \mathbf{a}_j, 1 \right) \nonumber \\
    &\leq \frac{1+\epsilon}{1-\epsilon} \widehat{\tau}_{i,j} \nonumber \\
    &\leq 3 \widehat{\tau}_{i,j}. \label{eqn:differencefor4.1.3_2}
\end{align}
Otherwise, i.e., if $\mathbf{a}_j \not\perp \mathrm{Ker} \left( \mathbf{M}_{i-1} \right)$,  $\widetilde{l}_j = 1 = \widehat{\tau}_{i,j}$. 
Thus we have $\widetilde{l}_j \le 3 \widehat{\tau}_{i,j}$ in any case.
Therefore the approximation size is $\mathrm{O}\left(\epsilon^{-2} \| \widetilde{\mathbf{l}} \|_1  \log d\right)=\mathrm{O} \left( d \epsilon^{-2} \log n \log d \right)$ with high probability.
\end{proof}

\subsubsection{Running time}

To prove \cref{theorem:scaled_sampling}, it remains to discuss the time complexity.

For each row $j$ in the $i$-th block,
we need to compute $\tau_j^{\widetilde{\mathbf{M}}_{i-1}}(\A)=\mathbf{a}_j^\top \left( \widetilde{\mathbf{M}}_{i-1}^\top \widetilde{\mathbf{M}}_{i-1} \right)^+ \mathbf{a}_j$.
This can be done efficiently with a well-known trick using the Johnson-Lindenstrauss lemma~\cite{DBLP:conf/pods/Achlioptas01}. 
Briefly speaking, the Johnson-Lindenstrauss lemma says that all-pair Euclidean distances between fixed $n$ vectors are preserved under a random projection  of the $n$ vectors onto an $\mathrm{O}\left( \log n \right)$ dimensional space. 
Define the random projection as $\mathbf{\Pi}$.
As $\widetilde{\mathbf{M}}_{i-1}$ has $\mathrm{O}\left( d \epsilon^{-2} \log n \log d \right)$ rows, $\mathbf{\Pi}$ is an $\mathrm{O}\left(\log n\right) \times \mathrm{O}\left( d \epsilon^{-2} \log n \log d \right)$ matrix. 

In the beginning of the $i$-th block, we compute the matrix $\mathbf{N}_i=\mathbf{\Pi} \widetilde{\mathbf{M}}_{i-1} \left( \widetilde{\mathbf{M}}_{i-1}^\top \widetilde{\mathbf{M}}_{i-1} \right)^+$.
This takes $\mathrm{O} \left( \left( d^\omega + d^2 \log n \right) \epsilon^{-2} \log n \log d \right)$ time, as $\mathbf{\Pi}\widetilde{\mathbf{M}}_{i-1}$ and $\left( \widetilde{\mathbf{M}}_{i-1}^\top \widetilde{\mathbf{M}}_{i-1} \right)^+$ can be computed in $\mathrm{O} \left( d^2 \epsilon^{-2} \log^2 n \log d \right)$ time and $\mathrm{O} \left( d^{\omega} \epsilon^{-2} \log n \log d \right)$ time respectively.
If we have $\mathbf{N}_i$, then we can compute an approximation of $\tau_j^{\widetilde{\mathbf{M}}_{i-1}} (\A)$ efficiently, since $\tau_j^{\widetilde{\mathbf{M}}_{i-1}} (\A)\approx \left\| \mathbf{N}_i \mathbf{a}_j \right\|_2^2$.
This takes $\mathrm{O} \left( \mathrm{nnz}\left( \mathbf{a}_j \right) \log n \right)$ time, where $\mathrm{nnz}\left( \mathbf{a}_j \right)$ is the number of nonzero entries in $\mathbf{a}_j$.
Since the number of blocks is $\mathrm{O}\left( \log n \right)$,
the total time complexity of \cref{algorithm:scaled_sampling} is $\mathrm{O}\left( \mathrm{nnz}(\mathbf{A}) \log n + \left( d^\omega + d^2 \log n \right) \epsilon^{-2} \right.$
$\left. \log^2 n \log d \right)$.

This completes the proof of \cref{theorem:scaled_sampling}.

\subsection{Fast $(1 \pm \epsilon)$-spectral approximation}
\label{subsection:less_memory}
In \cref{subsection:less_memory} and \cref{subsection:less_runtime_memory}, we improve \cref{algorithm:scaled_sampling} to obtain a time and space efficient algorithm.
In \cref{algorithm:scaled_sampling}, a $(1 \pm \epsilon)$-spectral approximation $\widetilde{\mathbf{M}}_{i-1}$ for $\mathbf{M}_{i-1}$ is used for computing a leverage score overestimate for the $i$-th block.
However, in the proof of \cref{theorem:scaled_sampling}~(see the inequalities \eqref{eqn:differencefor4.1.3_1} and \eqref{eqn:differencefor4.1.3_2}), in order to construct a $(1 \pm \epsilon)$-spectral approximation, a \emph{constant} spectral approximation for $\mathbf{M}_{i-1}$ suffices. Therefore, by computing leverage score overestimates with a constant spectral approximation instead of a $(1 \pm \epsilon )$-spectral approximation, we can reduce its running time and working memory.
Let $\mathsf{ConstApprox}$ denote a stream that runs a procedure to maintain a constant approximation. Then the algorithm is described as follows.


\begin{algorithm}[htbp]
\caption{$\mathsf{Improved Scaled Sampling} \left( \mathbf{A}, \epsilon \right)$}
\label{algorithm:improved_scaled_sampling}
\begin{algorithmic}
\STATE{{\bf Input:} a matrix $\mathbf{A} \in \mathbb{R}^{n \times d}$, an error parameter $\epsilon \in (0,1/2]$.}
\STATE{{\bf Output:} a $(1 \pm \epsilon)$-spectral approximation for $\mathbf{A}$.}
\STATE{Define $K = d \log d$, $c = 6 \epsilon^{-2} \log d$, $\alpha = \log_2 (n / K + 1) - 1$.}
\STATE{$\widetilde{\mathbf{A}}_0 \leftarrow \mathbf{O}$.}
\FOR{$j = 1, \dots, K$}
    \STATE{$\widetilde{\mathbf{A}}_j \leftarrow \begin{pmatrix} \widetilde{\mathbf{A}}_{j-1} \\ \mathbf{a}_j^\top \end{pmatrix}$.}
    \STATE{$\mathsf{ConstApprox}.\mathsf{add}(\mathbf{a}_j^\top)$}.
\ENDFOR
\FOR{$i = 1, \dots, \alpha$}
    \STATE{$\widetilde{\mathbf{M}}^c_{i-1} \leftarrow \mathsf{ConstApprox}.\mathsf{query}()$.}
    \STATE{$\tilde{l}_j \leftarrow \min\left(2 \tau_j^{\widetilde{\mathbf{M}}^c_{i-1}}\left( \mathbf{A} \right), 1\right)$.}
    \STATE{$p_j \leftarrow \min \left( c\tilde{l}_j, 1 \right)$.}
    \STATE{$\widetilde{\mathbf{A}}_j \leftarrow
    \left\{
        \begin{array}{cl}
            \begin{pmatrix} \widetilde{\mathbf{A}}_{j-1} \\ \mathbf{a}_j^\top /\sqrt{p_j} \end{pmatrix} & \text{with probability $p_j$,} \\
            \widetilde{\mathbf{A}}_{j-1} & \text{otherwise.} \\
        \end{array}
    \right. $
    }
    \STATE{$\mathsf{ConstApprox}.\mathsf{add}(\mathbf{a}_j^\top)$}.
\ENDFOR
\RETURN $\widetilde{\mathbf{A}}_n$.
\end{algorithmic}
\end{algorithm}

In \cref{algorithm:improved_scaled_sampling}, $\mathsf{ConstApprox}.\mathsf{add}(\mathbf{a}_j^\top)$ means that we add $\mathbf{a}_j^\top$ to the stream $\mathsf{ConstApprox}$.
Moreover, $\mathsf{ConstApprox}.\mathsf{query}()$ returns the current spectral approximation obtained by $\mathsf{ConstApprox}$. 
In the algorithm, we store $\widetilde{\mathbf{A}}_j$  in the output memory, as $\widetilde{\mathbf{A}}_j$ stores a family of output rows and is never referred.
In contrast, $\widetilde{\mathbf{M}}^c_{i-1}$ obtained by $\mathsf{ConstApprox}$ is stored in the working memory.
We prove that \cref{algorithm:improved_scaled_sampling} satisfies \cref{theorem:improved_scaled_sampling} when $\mathsf{ConstApprox}$ is set to $\mathsf{Scaled Sampling}\left( \mathbf{A}, 1/2 \right)$.



\begin{proof}[Proof of \cref{theorem:improved_scaled_sampling}]
\label{proof:improved_scaled_sampling}
Set $\mathsf{ConstApprox}$ in \cref{algorithm:improved_scaled_sampling} to $\mathsf{ScaledSampling}\left( \mathbf{A}, 1/2 \right)$.
We note that, for $1 \le i \le \alpha$, $\widetilde{\M}^c_{i-1}$ is a $\left( 1 \pm 1/2 \right)$-spectral approximation for $\widetilde{\M}_{i-1}$ with high probability from \cref{corollary:analysis_random_order_streams}.
Then, in a similar way to  \cref{algorithm:scaled_sampling}, we can prove that $\widetilde{\mathbf{A}}_n$ is a $(1 \pm \epsilon )$-spectral approximation, and that the approximation size of $\widetilde{\mathbf{A}}_n$ is bounded.
We omit the details.

It remains to show its time and space complexities.
Similarly to the proof of \cref{theorem:scaled_sampling}, 
we compute $\mathbf{N}_i = \mathbf{\Pi} \widetilde{\mathbf{M}}^c_{i-1} \left( \widetilde{\mathbf{M}}^{c \top}_{i-1} \widetilde{\mathbf{M}}^c_{i-1} \right)^+$ in the beginning of the $i$-th block, and $\tau_j^{\widetilde{\mathbf{M}}^c_{i-1}} (\A)\approx \left\| \mathbf{N}_i \mathbf{a}_j \right\|_2^2$ for each row $j$ in the $i$-th block.
Since $\widetilde{\mathbf{M}}^c_{i-1}$ now has $\mathrm{O}\left( d \log n \log d \right)$ rows, the time complexity to compute $\mathbf{N}_i$ is reduced to $\mathrm{O} \left( \left(d^\omega + d^2 \log n \right) \log n \log d \right)$ time.
 It takes $\mathrm{O} \left( \mathrm{nnz} (\A) \log n \right)$ time to compute $\left\| \mathbf{N}_i \mathbf{a}_j \right\|_2^2$ for all $i, j$.
Therefore, the total running time is $\mathrm{O} \left( \mathrm{nnz}(\A) \log n + \left(d^\omega + d^2 \log n \right) \log^2 n \log d \right)$ time. 
Regarding the space complexity, we only need to store $\mathrm{O} \left( d \log n \log d \right)$ rows for $\widetilde{\mathbf{M}}^c_i$ and an additional $\mathrm{O}\left( \log n \right)$ rows to keep $\mathbf{N}_i$ as the working memory.
\end{proof}

\subsection{Memory-efficient $(1 \pm \epsilon)$-spectral approximation}
\label{subsection:less_runtime_memory}
In this section, we achieve a further improvement in the working memory.
The required working space is $\mathrm{O} \left( d \log d \right)$ rows, which does not depend on $\epsilon$ and $n$. 
In the \cref{algorithm:improved_scaled_sampling}, instead of keeping a constant spectral approximation with $\mathsf{ScaledSampling} \left( \A, 1/2 \right)$, 
we run a spectral approximation algorithm in the semi-streaming setting in parallel to obtain a $(1 \pm 1/3)$-spectral approximation for $\mathbf{M}_{i-1}$. This is then used to compute a leverage score overestimate, which reduces the working memory space.

It is known as below that there exists an efficient semi-streaming algorithm for spectral approximation~\cite{DBLP:conf/soda/KyngPPS17}.

\begin{theorem} [Sparsification in the Semi-Streaming Setting \cite{DBLP:conf/soda/KyngPPS17}]
\label{theorem:resparsification}
    Let $\mathbf{A} \in \mathbb{R}^{n \times d}$ be a matrix, and $\epsilon \in (0, 1/2)$ be an error parameter.
    In the semi-streaming setting, we can construct, with high probability, a $(1 \pm \epsilon)$-spectral approximation for $\mathbf{A}$ with $\mathrm{O}(d \epsilon^{-2} \log d)$ rows by storing $\mathrm{O} (d \epsilon^{-2} \log d)$ rows in $\mathrm{O}\left( nd^{\omega-1} + nd \log \left(d \epsilon^{-1} \right) \right)$ time.
\end{theorem}
Note that the above algorithm maintains a $\left( 1 \pm \epsilon \right)$-spectral approximation at all times and runs in $\mathrm{O}\left(n \log^2 d \right)$ time if $\mathbf{A}$ is the incidence matrix of a graph.

Let $\mathsf{StreamSparsify}\left( \mathbf{A}, \epsilon \right)$ denote an algorithm which satisfies the above theorem. We prove that \cref{algorithm:improved_scaled_sampling} satisfies the following theorem  when $\mathsf{ConstApprox}$ is set to $\mathsf{StreamSparsify}\left( \mathbf{A}, 1/3 \right)$. 



\begin{theorem}
\label{theorem:improved_scaled_sampling_less_memory}
    Let $\epsilon \in (0, 1/2]$ be an error parameter, and $\X$ be a family of $n$ row vectors in $\mathbb{R}^d$.
    Then $\A \sim \mathcal{A}(\X)$ satisfies the following with high probability:
    If we set $\mathsf{ConstApprox}$ to $\mathsf{StreamSparsify}\left( \mathbf{A}, 1/3 \right)$,
    \cref{algorithm:improved_scaled_sampling} returns a $(1 \pm \epsilon)$-spectral approximation for $\mathbf{A}$ with $\mathrm{O}\left( d \epsilon^{-2} \log n \log d \right)$ rows with high probability.
    It consumes $\mathrm{O} \left( nd^{\omega-1} + nd \log d + \mathrm{nnz}(\A) \log n \right)$ time, where $\mathrm{nnz}(\A)$ is the number of non-zero entries in $\A$,  and stores $\mathrm{O}\left( d \log d \right)$ rows as the working memory and $\mathrm{O}\left( d \epsilon^{-2} \log n \log d \right)$ rows as the output memory.
\end{theorem}

\begin{proof}
\label{proof:improved_scaled_sampling_less_memory}
In a similar way to \cref{theorem:scaled_sampling} and \cref{theorem:improved_scaled_sampling}, we can prove that $\widetilde{\mathbf{A}}_n$ is a $(1 \pm \epsilon )$-spectral approximation, and that the approximation size of $\widetilde{\mathbf{A}}_n$ is bounded.
We omit the details.

It remains to show its time and space complexities. The semi-streaming algorithm consumes $\mathrm{O}\left( n d^{\omega-1} + nd \log d \right)$ time and stores $\mathrm{O} \left( d \log d \right)$ rows from \cref{theorem:resparsification} (when $\epsilon = 1/3$). We use the same notation as the proof of \cref{theorem:improved_scaled_sampling}. Since the number of rows in $\widetilde{\M}^c_{i-1}$ is $\mathrm{O} \left( d \log d \right)$, the time complexity to compute $\mathbf{N}_i$ is $\mathrm{O} \left( \left(d^\omega + d^2 \log n \right) \log d \right)$ time. It takes $\mathrm{O} \left( \mathrm{nnz} (\A) \log n \right)$ time to compute $\left\| \mathbf{N}_i \mathbf{a}_j \right\|_2^2$ for all $i, j$. Therefore, the total running time is $\mathrm{O} \left( nd^{\omega-1} + nd \log d + \mathrm{nnz}(\A) \log n \right)$. Regarding the space complexity, we only need to store $\mathrm{O} \left( d \log d \right)$ rows for $\widetilde{\mathbf{M}}^c_i$ and an additional $\mathrm{O}\left( \log n \right)$ rows to keep $\mathbf{N}_i$ as the working memory.
\end{proof}
Finally, for the case when an input matrix is the incidence matrix of a graph, the running time can be rewritten as follows.
\begin{corollary}
\label{corollary:online_random_order_setting}
    Let $\epsilon \in ( 0, 1/2 ]$ be an error parameter, and $\mathbf{X}$ be a family of $n$ row vectors in $\mathbb{R}^d$ corresponding to the incidence matrix of an undirected graph. If we set $\mathsf{ConstApprox}$ in \cref{algorithm:improved_scaled_sampling} to $\mathsf{StreamSparsify}\left( \mathbf{A}, 1/3 \right)$, \cref{algorithm:improved_scaled_sampling} runs in $\mathrm{O} \left( n \log n \right)$ time.
\end{corollary}
\begin{proof}
    We use the same notation as the proof of \cref{theorem:improved_scaled_sampling}.
    The semi-streaming algorithm consumes $\mathrm{O}\left( n \log^2 d \right)$ time.
    Since $\widetilde{\mathbf{M}}^{c\top}_{i-1} \widetilde{\mathbf{M}}^c_{i-1}$ is a Laplacian matrix, we compute the approximation of $\mathbf{N}_i$ by solving $\mathrm{O}\left( \log n \right)$ Laplacian systems. It consumes $\mathrm{O}\left( d \log^2 d \log n \right)$ time~\cite{DBLP:conf/focs/KoutisMP11}.
    Lastly, it takes $\mathrm{O} \left( \mathrm{nnz} \left( \A \right) \log n \right) = \mathrm{O} \left( n \log n \right)$ time to compute $\left\| \mathbf{N}_i \mathbf{a}_j \right\|_2^2$ for all $i, j$.
    Therefore, the total running time is $\mathrm{O} \left( n \log n \right)$ time.
\end{proof}

\section{Lower bound in the online random order setting}
\label{section:optimality}

In this section, we prove \cref{theorem:optimal}, that is, in order to maintain a $(1 \pm \epsilon)$-spectral approximation in the online random order setting, any randomized algorithm requires to keep $\Omega \left( d \epsilon^{-2} \log n \right)$ rows in the worst case.
Recall that the lower bound on the approximation size without any restriction is $\Omega(d \epsilon^{-2})$~\cite{DBLP:conf/stoc/BatsonSS09}, and thus the online random order setting suffers an additional $\log n$ factor.
We assume the adversary who knows only how an algorithm works and does not know any result after we run the algorithm.
In other words, the worst input stream is determined in advance.

We define an input family of vectors $\X^\ast$ as follows.
Let $K_d$ be a complete graph on $d$ vertices.
Then $\X^\ast$ is defined to be the incidence matrix of $n/\binom{d}{2}K_d$, where, for a graph $G$ and a non-negative number $\alpha$, $\alpha G$ is a graph obtained from $G$ by making $\alpha -1$ copies of each edge.
That is, $\X^\ast$ has $n/\binom{d}{2}$ copies of each row in the incidence matrix $\B_{K_d}$.
We will prove that the family $\X^\ast$ gives a lower bound in \cref{theorem:optimal}.

We first show in \cref{lemma:approximate_complete_graph} below that there exists a constant $D$ such that $\A\sim\mathcal{A}(\X^\ast)$ satisfies with high probability that the submatrix $\mathbf{A}_D$ , which is the matrix composed of the first $D$ rows in $\mathbf{A}$, is a $(1 \pm \epsilon)$-spectral approximation for $\mathbf{B}_{D / \binom{d}{2} K_d}$.
This can be done by regarding sampling in the online random order setting as sampling without replacement from a finite population.
\begin{lemma}[Tail Bound for the Hypergeometric Distribution \cite{hypergeo}] \label{lemma:hyper}
  Let $C$ be a set of $M$ elements that contains $K$ $1$'s and $M-K$ $0$'s.
  Let $X_1, \dots, X_m$ denote the values drawn from $C$ without replacement. Define $S_i \overset{\mathrm{def}}{=} \sum_{j=1}^{i} X_j$ and $\mu \overset{\mathrm{def}}{=} K / M$. Then for all $t > 0$, we have
  \begin{displaymath}
    \mathbb{P} \left( | S_m - m \mu | \geq m t \right) \leq 2 \exp \left( - \frac{2 m t ^ 2}{1 - f^*_m} \right),
  \end{displaymath}
  where $f^*_m = (m-1) / M$.
\end{lemma}
\begin{lemma}
\label{lemma:approximate_complete_graph}
    Let $\mathbf{A} \in \mathbb{R}^{n \times d}$ be a matrix chosen from the uniform distribution $\mathcal{A}(\X^\ast)$, and $\epsilon \in (0, 1)$ be an error parameter. 
    Set $D= d^4 \epsilon^{-2} \log d$ and $\alpha = D / \binom{d}{2}$.
    Then, with high probability, $\mathbf{A}_D$ is a $(1 \pm \epsilon)$-spectral approximation for $\mathbf{B}_{\alpha K_d}$.
\end{lemma}
\begin{proof}
   Let $S_D(e)$ be the number of an edge $e$ of $K_d$ in $\mathbf{A}_D$.
   We show $(1-\epsilon) \alpha \leq S_D(e) \leq (1+\epsilon)\alpha$ with high probability, where we note that $\alpha$ is the expected value of $S_D(e)$.
   We apply \cref{lemma:hyper}, where $M=n$, $m=D$, $\mu  = 1 / \binom{d}{2}$, and $S_m = S_D(e)$.
   Setting $t = \epsilon /\binom{d}{2} = \epsilon \alpha /D$, we have
  \begin{align}
    \mathbb{P} \left( \left| S_D(e) - \alpha \right| \geq \epsilon \alpha  \right)
    &\le 2 \exp \left( - \frac{N}{N-D+1} \cdot \frac{2 d^4 \log d}{\epsilon ^2} \left( \frac{\epsilon}{\binom{d}{2}} \right)^2 \right) \nonumber \\
    &\le 2 \exp (-8 \log d) \nonumber \\
    &= \frac{2}{d^8}. \nonumber 
  \end{align}
  Thus $(1-\epsilon) \alpha \leq S_D(e) \leq (1+\epsilon)\alpha$ with probability at least $1-2/d^8$.
  Since $K_d$ has $O(d^2)$ edges, the above inequality holds for all edges in $K_d$ with high probability, by taking a union bound.

  Therefore, the graph corresponding to $\mathbf{A}_D$ satisfies that the weight of every edge is between $(1 - \epsilon) \alpha$ and $(1 + \epsilon) \alpha$ with high probability, which means that $\mathbf{A}_D$ is a $(1 \pm \epsilon)$-spectral approximation for $\mathbf{B}_{\alpha K_d}$.
\end{proof}

Let $R$ be an algorithm that returns a $(1 \pm \epsilon)$-spectral approximation. 
Define $s_i = 2^i D$ for $i=0,1,\dots, \left\lfloor \log_2 (n/D) \right\rfloor$.
Since $R$ returns a $(1 \pm \epsilon)$-spectral approximation for any instance, $R$ has to keep a $(1\pm \epsilon)$-spectral approximation for $\A_j$ for all $j=1,\dots, n$.
Hence, for all $i=0,\dots, \left\lfloor \log_2 (n/D) \right\rfloor$, our matrix $\mathbf{\widetilde{A}}_{s_i}$ is a $(1 \pm \epsilon)$-spectral approximation for $\mathbf{A}_{s_i}$.
On the other hand, it follows from  \cref{lemma:approximate_complete_graph} that $\A_{s_i}$ must be a $(1 \pm \epsilon)$-spectral approximation for $\mathbf{B}_{\alpha_i K_d}$ for some $\alpha_i$.
They imply that we have to sample $\Omega (d\epsilon^{-2})$ rows between $s_i$-th and $s_{i+1}$-th rows.


 \begin{proof}[Proof of \cref{theorem:optimal}]
 Set $\mathbf{X}$ to the family $\X^\ast$ defined above, and $A \sim \mathcal{A}(\X)$.
 We assume that $\epsilon^2 n > d^4 \log d$ and $\epsilon^2 d > c_1$ and $\exp \left( d^{c_2} \right) > n$ for fixed positive constants $c_1$ and $c_2$.
 Define $s_i = 2^i D$, $\alpha_i = s_i / \binom{d}{2}$ for $i=0, 1,\dots, \left\lfloor \log_2 (n/D) \right\rfloor$, where $D=d^4 \epsilon^{-2}\log d$.
 We remark that since $\epsilon^2 n > d^4 \log d$, it holds that $\left\lfloor \log_2 (n/D) \right\rfloor \geq 0$.
 
  Let $\tA_j$ be a matrix that an algorithm $R$ maintains after the $j$-th row arrived. 
  Since $R$ returns a $(1 \pm \epsilon)$-spectral approximation for any matrix, $\widetilde{\mathbf{A}}_{s_i}$ is a $(1 \pm \epsilon)$-spectral approximation for $\mathbf{A}_{s_i}$ for all $i = 0,1,\dots, \left\lfloor \log_2 (n/D) \right\rfloor$:
  \begin{align}
    (1 - \epsilon) \mathbf{A}_{s_i}^\top \mathbf{A}_{s_i} \preceq &\widetilde{\mathbf{A}}_{s_i}^\top \widetilde{\mathbf{A}}_{s_i} \preceq (1 + \epsilon) \mathbf{A}_{s_i}^\top \mathbf{A}_{s_i}. \label{equation:ieq2}
  \end{align}

   We first suppose that for all $i=0,1,\dots, \left\lfloor \log_2 (n/D) \right\rfloor$, $\mathbf{A}_{s_i}$ is a $(1 \pm \epsilon)$-spectral approximation for $\mathbf{B}_{\alpha_i K_d}$ simultaneously:
 \begin{align}
     (1 - \epsilon) \mathbf{L}_{\alpha_i K_d} \preceq \mathbf{A}_{s_i}^\top \mathbf{A}_{s_i} \preceq (1 + \epsilon) \mathbf{L}_{\alpha_iK_d}. \label{equation:ieq1}
 \end{align}
  (We evaluate the probability that it holds in the end of the proof.)
  From \cref{equation:ieq2} and \cref{equation:ieq1}, we obtain
  \begin{align}
    (1 - 3 \epsilon) \mathbf{L}_{\alpha_i K_d} \preceq &\widetilde{\mathbf{A}}_{s_i}^\top \widetilde{\mathbf{A}}_{s_i} \preceq (1 + 3 \epsilon) \mathbf{L}_{\alpha_i K_d}. \label{equation:ieq3}
  \end{align}
  Similarly, we obtain
  \begin{align}
    (1 - 3\epsilon) \mathbf{L}_{\alpha_{i+1} K_d} \preceq &\widetilde{\mathbf{A}}_{s_{i+1}}^\top \widetilde{\mathbf{A}}_{s_{i+1}} \preceq (1 + 3\epsilon) \mathbf{L}_{\alpha_{i+1} K_d}. \label{equation:ieq4}
  \end{align}
  Subtracting \cref{equation:ieq3} from \cref{equation:ieq4} in both sides, we obtain
  \begin{displaymath}
    (1 - 9 \epsilon) \mathbf{L}_{\alpha_i K_d} \preceq \widetilde{\mathbf{A}}_{s_{i+1}}^\top \widetilde{\mathbf{A}}_{s_{i+1}} - \widetilde{\mathbf{A}}_{s_i}^\top \widetilde{\mathbf{A}}_{s_i} \preceq (1 + 9 \epsilon) \mathbf{L}_{\alpha_i K_d}.
  \end{displaymath}
  Hence the rows sampled between $s_i$-th and $s_{i+1}$-th rows in algorithm $R$ form a $(1 \pm 9 \epsilon)$-spectral approximation for $\mathbf{B}_{\alpha_i K_d}$.
  Since it requires $\Omega \left( d \epsilon^{-2} \right)$ rows to construct a $\left( 1 \pm \mathrm{O} (\epsilon) \right)$-spectral approximation for $\mathbf{B}_{K_d}$ when $\epsilon^2 d > c_1$~\cite{DBLP:conf/stoc/BatsonSS09}, the number of rows sampled between $s_i$-th and $s_{i+1}$-th rows is $\Omega \left( d \epsilon^{-2} \right)$.
  Aggregating it for all $i=0,1,\dots,  \left\lfloor \log_2 (n/D) \right\rfloor - 1$, we see that the approximation size is $\Omega\left( d \epsilon^{-2} \log n \right)$.
  
  Finally, we evaluate the probability that for all $i=0, 1,\dots, \left\lfloor \log_2 (n/D) \right\rfloor$, $\mathbf{A}_{s_i}$ is a $(1 \pm \epsilon)$-spectral approximation for $\mathbf{B}_{\alpha_i K_d}$.
  By \cref{lemma:approximate_complete_graph}, $\mathbf{A}_{s_i}$ is with high probability a $(1 \pm \epsilon)$-spectral approximation for $\mathbf{B}_{\alpha_i K_d}$.
  Since we assume that $\exp \left( d^{c_2} \right) > n$, it holds that $\log_2 \left(n/D\right)=O(d)$.
  Taking a union bound, $\mathbf{A}_{s_i}$ is with high probability a $(1 \pm \epsilon)$-spectral approximation for $\mathbf{B}_{\alpha_i K_d}$ for all $i=0,1,\dots, \left\lfloor \log_2 \left( n / D \right) \right\rfloor$.
  In summary, $R$ must sample $\Omega \left( d \epsilon^{-2} \log n \right)$ rows with high probability.
\end{proof}
\appendix
\section{Optimal Algorithms in the Online and Random Order Settings}
\label{section:appendix}
In this section, we show that there exists an optimal spectral approximation algorithm in the online setting by applying a technique in \cite{DBLP:conf/approx/CohenMP16} to \cref{algorithm:online_row_sampling}.
This also leads to an optimal algorithm in the online random order setting.

\subsection{Online Setting}

In the algorithm below, we maintain an upper barrier and a lower barrier of our matrix as guideposts.
In the beginning of the $i$-th iteration, we have an upper barrier $\mathbf{B}_{i-1}^\mathrm{U}$ and a lower barrier $\mathbf{B}_{i-1}^\mathrm{L}$ of our matrix $\widetilde{\mathbf{A}}_{i-1}^\top \widetilde{\mathbf{A}}_{i-1}$.
That is, $\widetilde{\mathbf{A}}_{i-1}^\top \widetilde{\mathbf{A}}_{i-1}$ is located between $\mathbf{B}_{i-1}^\mathrm{L}$ and $\mathbf{B}_{i-1}^\mathrm{U}$.
Then we compute differences between $\widetilde{\mathbf{A}}_{i-1}^\top \widetilde{\mathbf{A}}_{i-1}$ and $\mathbf{B}_{i-1}^\mathrm{U}$~($\mathbf{B}_{i-1}^\mathrm{L}$, resp.), and 
sample the $i$-th row based on these differences.
In the end of the $i$-th iteration, we update $\mathbf{B}_{i-1}^\mathrm{U}$ and $\mathbf{B}_{i-1}^\mathrm{L}$ so that $\widetilde{\mathbf{A}}_{i}^\top \widetilde{\mathbf{A}}_{i}$ is located between them.

While the algorithm is optimal for the approximation size, we need to store $\mathrm{O}\left( d^2 \right)$ rows, which is worse than previous algorithms and runs slowly compared to \cref{algorithm:scaled_sampling}.
\begin{algorithm}[htbp]
\caption{$\mathsf{Optimal Online Row Sampling} \left( \mathbf{A}, \epsilon \right)$}
\label{algorithm:optimal}
\begin{algorithmic}
\STATE{{\bf Input:} a matrix $\mathbf{A} \in \mathbb{R}^{n \times d}$, an error parameter $\epsilon \in (0,1)$.}
\STATE{{\bf Output:} a $(1 \pm \epsilon)$-spectral approximation for $\mathbf{A}$.}
\STATE{Define $c_\mathrm{U} = 2 / \epsilon + 1$, $c_\mathrm{L} = 3 / \epsilon - 1$.}
\STATE{$\widetilde{\mathbf{A}}_0 \leftarrow \mathbf{O}, \mathbf{B}_0^\mathrm{U} \leftarrow \mathbf{O}, \mathbf{B}_0^\mathrm{L} \leftarrow \mathbf{O}$.}
\FOR{$i = 1, \dots, n$}
  \STATE{Let $\mathbf{X}^\mathrm{U}_{i-1} := \left( \mathbf{B}^\mathrm{U}_{i-1} - \widetilde{\mathbf{A}}^\top_{i-1} \widetilde{\mathbf{A}}_{i-1}\right) + \mathbf{a}_i\mathbf{a}_i^\top, \mathbf{X}^\mathrm{L}_{i-1} := \left( \widetilde{\mathbf{A}}^\top_{i-1} \widetilde{\mathbf{A}}_{i-1} - \mathbf{B}^\mathrm{L}_{i-1}\right) + \mathbf{a}_i\mathbf{a}_i^\top$.}
\STATE{$p_i \leftarrow \min \left( c_\mathrm{U} \mathbf{a}_i^\top \left( \mathbf{X}^\mathrm{U}_{i-1} \right)^+ \mathbf{a}_i + c_\mathrm{L} \mathbf{a}_i^\top \left( \mathbf{X}^\mathrm{L}_{i-1}\right)^+ \mathbf{a}_i  , 1 \right)$.}
  \STATE{$\widetilde{\mathbf{A}}_i \leftarrow
  \left\{
    \begin{array}{cl}
        \begin{pmatrix} \widetilde{\mathbf{A}}_{i-1} \\ \mathbf{a}_i^\top /\sqrt{p_i} \end{pmatrix} & \text{with probability $p_i$,} \\
        \widetilde{\mathbf{A}}_{i-1} & \text{otherwise.} \\
    \end{array}
  \right. $}
  \STATE{$\mathbf{B}^\mathrm{U}_i = \mathbf{B}^\mathrm{U}_{i-1} + \left(1 + \epsilon \right)\mathbf{a}_i\mathbf{a}_i^\top, \mathbf{B}^\mathrm{L}_i = \mathbf{B}^\mathrm{L}_{i-1} + \left(1 - \epsilon \right) \mathbf{a}_i \mathbf{a}_i^\top$.}
\ENDFOR
\RETURN $\widetilde{\mathbf{A}}_n$.
\end{algorithmic}
\end{algorithm}

\begin{lemma}
\label{lemma:optimal_appendix}
Let $\mathbf{A} \in \mathbb{R}^{n \times d}$ be a matrix and $\epsilon \in (0, 1)$ be an error parameter. Then, in the online setting, \cref{algorithm:optimal} returns a $(1 \pm \epsilon)$-spectral approximation for $\mathbf{A}$ with $\mathrm{O}\left( \epsilon^{-2} \sum_i \tau_i^{\mathbf{A}_{i-1}} \left(\mathbf{A}\right) \right)$rows in expectation.
\end{lemma}
Before starting the proof, we mention the following lemma.

\begin{lemma}
\label{lemma:optimal_counterpart}
    Let $\mathbf{X} \in \mathbb{R}^{d \times d}$ be a $\mathrm{PSD}$ matrix, and $\mathbf{u}$, $\mathbf{v} \in \mathbb{R}^d$ be vectors such that $\mathbf{u} \perp \mathrm{Ker}(\mathbf{X})$ and $\mathbf{u}^\top \mathbf{X}^+ \mathbf{u} = 1$.
    For $a, b \in \mathbb{R}$ with $a, b \neq 1$, define the random variable $\mathbf{X}^{'}$ to be $\mathbf{X}-a \mathbf{u}\mathbf{u}^\top$ with probability $p$ and $\mathbf{X}-b\mathbf{u}\mathbf{u}^\top$ otherwise. Then
    \begin{displaymath}
        \mathbb{E} \left[ \mathbf{v}^\top \mathbf{X}^{'+} \mathbf{v} - \mathbf{v}^\top \mathbf{X}^{+} \mathbf{v} \right] = \left( \mathbf{v}^\top \mathbf{X}^{+} \mathbf{u} \right)^2 \frac{pa+(1-p)b-ab}{(1-a)(1-b)}.
    \end{displaymath}
\end{lemma}
\begin{proof}
    Directly calculating by \cref{prop:sherman-morrison_Moore-Penrose}, we obtain
    \begin{align}
    \mathbb{E} \left[ \mathbf{v}^\top \mathbf{X}^{'+} \mathbf{v} - \mathbf{v}^\top \mathbf{X}^{+} \mathbf{v} \right]
        &= \mathbf{v}^\top \left( p \cdot \frac{a\mathbf{X}^+\mathbf{u}\mathbf{u}^\top\mathbf{X}^+}{1 - a} + (1-p) \cdot \frac{b\mathbf{X}^+\mathbf{u}\mathbf{u}^\top\mathbf{X}^+}{1 - b} \right) \mathbf{v} \nonumber \\
        &= \left( \mathbf{v}^\top \mathbf{X}^{+} \mathbf{u} \right)^2 \frac{pa+(1-p)b-ab}{(1-a)(1-b)}. \nonumber
    \end{align}
\end{proof}
\begin{proof}[Proof of \cref{lemma:optimal_appendix}]
First of all, we will prove that, in the end of the $i$-th iteration, $\widetilde{\mathbf{A}}_i^\top \widetilde{\mathbf{A}}_i$ is located between the upper barrier $\mathbf{B}_i^\mathrm{U}$ and the lower barrier $\mathbf{B}_i^\mathrm{L}$:
\begin{equation}\label{eq:app1}
    \mathbf{B}_i^\mathrm{L} \preceq \widetilde{\mathbf{A}}_i^\top \widetilde{\mathbf{A}}_i \preceq \mathbf{B}_i^\mathrm{U}
\end{equation}
Define the gap to the upper bound as $\widehat{\mathbf{X}}_i^\mathrm{U} = \mathbf{B}^\mathrm{U}_{i-1} - \widetilde{\mathbf{A}}^\top_{i-1} \widetilde{\mathbf{A}}_{i-1} = \mathbf{X}_i^\mathrm{U} - \mathbf{a}_{i+1}\mathbf{a}_{i+1}^\top$ and the gap to the lower bound as $\widehat{\mathbf{X}}_i^\mathrm{L} =  \widetilde{\mathbf{A}}^\top_{i-1} \widetilde{\mathbf{A}}_{i-1} - \mathbf{B}^\mathrm{L}_{i-1} = \mathbf{X}_i^\mathrm{L} - \mathbf{a}_{i+1}\mathbf{a}_{i+1}^\top$.
We say that a $\mathrm{PSD}$ matrix $\mathbf{M}$ \textit{increases} if $\mathbf{M}$ is changed to $\mathbf{M}'$ with $\mathbf{M} \preceq \mathbf{M}'$.

When $p_i = 1$, \cref{eq:app1} holds, since the gaps $\widehat{\mathbf{X}}_{i-1}^\mathrm{U}$ and $\widehat{\mathbf{X}}_{i-1}^\mathrm{L}$ strictly increase. When $p_i < 1$, $\widehat{\mathbf{X}}_{i-1}^\mathrm{U}$ can decrease by $\left( \mathbf{a}_i\mathbf{a}_i^\top \right) / p_i - (1 + \epsilon) \mathbf{a}_i \mathbf{a}_i^\top$. 
As $p_i > \mathbf{a}_i^\top \left( \mathbf{X}_{i-1}^\mathrm{U} \right)^+ \mathbf{a}_i$ and $\mathbf{a}_i\mathbf{a}_i^\top / p_i \prec \mathbf{X}_{i-1}^\mathrm{U}$, the amount of the decrease is at most
\begin{displaymath}
    \frac{\mathbf{a}_i\mathbf{a}_i^\top}{p_i} - (1 + \epsilon) \mathbf{a}_i \mathbf{a}_i^\top \prec \frac{\mathbf{a}_i\mathbf{a}_i^\top}{p_i} - \mathbf{a}_i \mathbf{a}_i^\top \prec \widehat{\mathbf{X}}_{i-1}^\mathrm{U}.
\end{displaymath}
Thus $\widetilde{\mathbf{A}}_i$ does not exceed the upper bound after the rank-1 update. We can also confirm the condition for the lower bound analogously. 

Since $\mathbf{B}_i^\mathrm{L}$ and $\mathbf{B}_i^\mathrm{U}$ are $(1\pm \epsilon)$-spectral approximations for $\A_i$, respectively, we see that, for all $i$, $\widetilde{\mathbf{A}}_i^\top \widetilde{\mathbf{A}}_i$ is located between $(1 \pm \epsilon)$ multiplicative bounds.
Hence \cref{algorithm:optimal} returns a $(1 \pm \epsilon)$-spectral approximation for $\mathbf{A}$. 

In the latter part of the proof, we bound the approximation size of $\widetilde{\mathbf{A}}_n$. We define $\mathbf{Y}_{i,j}^\mathrm{U}$, $\mathbf{Y}_{i,j}^\mathrm{L}$ as follows:
\begin{align}
    \mathbf{Y}_{i,j}^\mathrm{U} &= \frac{\epsilon}{2} \mathbf{A}^\top_i \mathbf{A}_i + \left( 1 + \frac{\epsilon}{2} \right) \mathbf{A}^\top_j \mathbf{A}_j - \widetilde{\mathbf{A}}_j^\top \widetilde{\mathbf{A}}_j + \mathbf{a}_{i+1} \mathbf{a}_{i+1}^\top, \nonumber \\
    \mathbf{Y}_{i,j}^\mathrm{L} &= \widetilde{\mathbf{A}}_j^\top \widetilde{\mathbf{A}}_j + \frac{\epsilon}{2} \mathbf{A}^\top_i \mathbf{A}_i - \left( 1 - \frac{\epsilon}{2} \right) \mathbf{A}^\top_j \mathbf{A}_j + \mathbf{a}_{i+1} \mathbf{a}_{i+1}^\top. \nonumber
\end{align}
Notice that $\mathbf{Y}_{i,i}^\mathrm{U} = \mathbf{X}_i^\mathrm{U}$  and $\mathbf{Y}_{i,i}^\mathrm{L} = \mathbf{X}_i^\mathrm{L}$ for all $i$, and $\mathbf{Y}_{i,j}^\mathrm{U} \succeq \mathbf{X}_j^\mathrm{U}$ and $\mathbf{Y}_{i,j}^\mathrm{L} \succeq \mathbf{X}_j^\mathrm{L}$ for any $j \le i$. We show the following two monotonicities for $j < i-1$:

\begin{claim}\label{clm:1}
\begin{align}
    \mathbb{E}\left[ \mathbf{a}_i^\top \left( \mathbf{Y}_{i-1,j+1}^\mathrm{U} \right)^+ \mathbf{a}_i \right] \le \mathbb{E}\left[ \mathbf{a}_i^\top \left( \mathbf{Y}_{i-1,j}^\mathrm{U} \right)^+ \mathbf{a}_i \right], \label{equation:monotonicity1} \\
    \mathbb{E}\left[ \mathbf{a}_i^\top \left( \mathbf{Y}_{i-1,j+1}^\mathrm{L} \right)^+ \mathbf{a}_i \right] \le \mathbb{E}\left[ \mathbf{a}_i^\top \left( \mathbf{Y}_{i-1,j}^\mathrm{L} \right)^+ \mathbf{a}_i \right]. \label{equation:monotonicity2}
\end{align}
\end{claim}
\begin{proof}[Proof of \cref{clm:1}]
If $p_{j+1} = 1$, then $\mathbf{Y}_{i-1,j+1}^\mathrm{U}$ and $\mathbf{Y}_{i-1,j+1}^\mathrm{L}$ increase:
\begin{displaymath}
    \mathbf{Y}_{i-1,j+1}^\mathrm{U} = \mathbf{Y}_{i-1,j}^\mathrm{U} + \frac{\epsilon}{2} \mathbf{a}_{j+1} \mathbf{a}_{j+1}^\top, \hspace{10pt} \mathbf{Y}_{i-1,j+1}^\mathrm{L} = \mathbf{Y}_{i-1,j}^\mathrm{L} + \frac{\epsilon}{2} \mathbf{a}_{j+1} \mathbf{a}_{j+1}^\top.
\end{displaymath}
Hence two monotonicities hold.

Next assume that $p_{j+1}<1$, where $\mathbf{a}_{j+1} \perp \mathrm{Ker} \left( \mathbf{X}_j^\mathrm{U} \right)$ and $\mathbf{a}_{j+1} \perp \mathrm{Ker} \left( \mathbf{X}_j^\mathrm{L} \right)$ hold.
First of all, since for $j \le i$ $\mathbf{Y}_{i,j}^\mathrm{U} \succeq \mathbf{X}_j^\mathrm{U}$ and $\mathbf{Y}_{i,j}^\mathrm{L} \succeq \mathbf{X}_j^\mathrm{L}$ hold, we have
\begin{align}
    p_{j+1} \ge c_\mathrm{U} \mathbf{a}_{j+1}^\top \left( \mathbf{X}_j^\mathrm{U} \right)^+ \mathbf{a}_{j+1} \ge c_\mathrm{U} \mathbf{a}_{j+1}^\top \left( \mathbf{Y}_{i-1,j}^\mathrm{U} \right)^+ \mathbf{a}_{j+1}, \nonumber \\
    p_{j+1} \ge c_\mathrm{L} \mathbf{a}_{j+1}^\top \left( \mathbf{X}_j^\mathrm{L} \right)^+ \mathbf{a}_{j+1} \ge c_\mathrm{L} \mathbf{a}_{j+1}^\top \left( \mathbf{Y}_{i-1,j}^\mathrm{L} \right)^+ \mathbf{a}_{j+1}. \nonumber
\end{align}
We define the vector $\mathbf{w}_{j+1} \overset{\mathrm{def}}{=}  \mathbf{a}_{j+1} / \sqrt{p_{j+1}}$ and get:
\begin{align}
    \mathbf{w}_{j+1}^\top \left( \mathbf{Y}_{i-1,j}^\mathrm{U} \right)^+ \mathbf{w}_{j+1} \le \frac{1}{c_\mathrm{U}}, \hspace{10pt}
    \mathbf{w}_{j+1}^\top \left( \mathbf{Y}_{i-1,j}^\mathrm{L} \right)^+ \mathbf{w}_{j+1} \le \frac{1}{c_\mathrm{L}}. \label{equation:s_inequlity}
\end{align}
Moreover, we define $s_{j+1}^\mathrm{U}$, $s_{j+1}^\mathrm{L}$, $\mathbf{u}_{j+1}^\mathrm{U}$, $\mathbf{u}_{j+1}^\mathrm{L}$ as follows:
\begin{align}
    s_{j+1}^\mathrm{U} &\overset{\mathrm{def}}{=} \mathbf{w}_{j+1}^\top \left( \mathbf{Y}_{i-1,j}^\mathrm{U} \right)^+ \mathbf{w}_{j+1}, \hspace{10pt}
    s_{j+1}^\mathrm{L} \overset{\mathrm{def}}{=} \mathbf{w}_{j+1}^\top \left( \mathbf{Y}_{i-1,j}^\mathrm{L} \right)^+ \mathbf{w}_{j+1}, \nonumber \\
    \mathbf{u}_{j+1}^\mathrm{U} &\overset{\mathrm{def}}{=} \frac{\mathbf{w}_{j+1}}{\sqrt{s_{j+1}^\mathrm{U}}}, \hspace{10pt}
    \mathbf{u}_{j+1}^\mathrm{L} \overset{\mathrm{def}}{=} \frac{\mathbf{w}_{j+1}}{\sqrt{s_{j+1}^\mathrm{L}}}. \nonumber
\end{align}

Applying \cref{lemma:optimal_counterpart} where $\mathbf{X}=\mathbf{Y}_{i-1,j}^\mathrm{U}$, $\mathbf{u}= \mathbf{u}_{j+1}^\mathrm{U}$, $\mathbf{v}=\mathbf{a}_i$, $a=s_{j+1}^\mathrm{U} \left(1 - p_{j+1} \left( 1 + \epsilon / 2 \right) \right)$, $b=-s_{j+1}^\mathrm{U} p_{j+1} \left( 1 + \epsilon / 2 \right)$, and $p=p_{j+1}$, respectively, we have
\begin{align}
    \mathbb{E} \left[ \mathbf{a}_i^\top \left( \mathbf{Y}_{i-1,j}^{\mathrm{U}'} \right)^+ \mathbf{a}_i - \mathbf{a}_i^\top \left( \mathbf{Y}_{i-1,j}^\mathrm{U} \right)^+ \mathbf{a}_i \right]
    = \left( \mathbf{a}_i^\top \left( \mathbf{Y}_{i-1,j}^\mathrm{U} \right)^+ \mathbf{u}_{j+1}^\mathrm{U} \right)^2
    \frac{pa+(1-p)b-ab}{(1-a)(1-b)} \label{equation:optimal_compicated}
\end{align}
where
\begin{align}
    \mathbf{Y}_{i-1,j}^{\mathrm{U}'}\overset{\mathrm{def}}{=}
    \left\{
    \begin{array}{cl}
            \mathbf{Y}_{i-1,j}^\mathrm{U} - s_{j+1}^\mathrm{U} \left(1 - p_{j+1} \left( 1 + \epsilon / 2 \right) \right) \mathbf{u}_{j+1}^\mathrm{U} \mathbf{u}_{j+1}^{\mathrm{U}\top} & \text{with probability } p_{j+1}, \\
            \mathbf{Y}_{i-1,j}^\mathrm{U} + s_{j+1}^\mathrm{U} p_{j+1} \left( 1 + \epsilon / 2 \right) \mathbf{u}_{j+1}^\mathrm{U} \mathbf{u}_{j+1}^{\mathrm{U}\top} & \text{otherwise.}
        \end{array}
    \right. \nonumber
\end{align}
From the inequality \eqref{equation:s_inequlity}, $s_{j+1}^\mathrm{U} \le 1 / c_\mathrm{U} < \epsilon / 2$ holds. 

Concerning the numerator in \cref{equation:optimal_compicated}, the following inequality holds:
\begin{align}
    p a + (1-p) b - ab
    &= p_{j+1}s_{j+1}^\mathrm{U} \left(1 - p_{j+1} \left( 1 + \frac{\epsilon}{2} \right) \right) - (1-p_{j+1})s_{j+1}^\mathrm{U} p_{j+1} \left( 1 + \frac{\epsilon}{2} \right) \nonumber \\
    &\hspace{20pt} + s_{j+1}^\mathrm{U} \left(1 - p_{j+1} \left( 1 + \frac{\epsilon}{2} \right) \right) s_{j+1}^\mathrm{U} p_{j+1} \left( 1 + \frac{\epsilon}{2} \right) \nonumber \\
    &= - \frac{\epsilon}{2} p_{j+1} s_{j+1}^\mathrm{U} + \left( 1 + \frac{\epsilon}{2} \right) p_{j+1} \left( s_{j+1}^\mathrm{U} \right)^2 - \left( 1 + \epsilon + \frac{\epsilon^2}{4} \right) p_{j+1}^2 \left( s_{j+1}^\mathrm{U} \right)^2 \nonumber \\
    &\le \frac{\epsilon}{2} p_{j+1} \left( s_{j+1}^\mathrm{U} \right)^2 - \left( 1 + \epsilon + \frac{\epsilon^2}{4} \right) p_{j+1}^2 \left( s_{j+1}^\mathrm{U} \right)^2 \nonumber \\
    &\le 0. \label{equation:optimal1}
\end{align}
Concerning the denominator in \cref{equation:optimal_compicated}, the following inequality holds:
\begin{align}
    (1-a)(1-b) &= \left( 1 - s_{j+1}^\mathrm{U} \left(1 - p_{j+1} \left( 1 + \frac{\epsilon}{2} \right) \right) \right) \left(1 + s_{j+1}^\mathrm{U} p_{j+1} \left( 1 + \frac{\epsilon}{2} \right) \right) \nonumber \\
    &= \left(1 - s_{j+1}^\mathrm{U} + p_{j+1} s_{j+1}^\mathrm{U} + \frac{\epsilon}{2} p_{j+1} s_{j+1}^\mathrm{U} \right) \left( 1 + p_{j+1} s_{j+1}^\mathrm{U} + \frac{\epsilon}{2} p_{j+1} s_{j+1}^\mathrm{U} \right) \nonumber \\
    &> 0 \label{equation:optimal2}
\end{align}
In addition, $\mathbf{Y}_{i-1,j}^{\mathrm{U}'} = \mathbf{Y}_{i-1,j+1}^\mathrm{U}$ holds, since it follows from the definitions of $\mathbf{w}_{j+1}$, $s_{j+1}^\mathrm{U}$, and $\mathbf{u}_{j+1}^\mathrm{U}$ that
\begin{align}
    \mathbf{Y}_{i-1,j}^{\mathrm{U}'}
    &= 
    \left\{
    \begin{array}{cl}
            \mathbf{Y}_{i-1,j}^\mathrm{U} + \left(1 + \epsilon / 2 \right)  \mathbf{a}_{j+1} \mathbf{a}_{j+1}^\top - \left( \mathbf{a}_{j+1}\mathbf{a}_{j+1}^\top \right) / p_{j+1} & \text{with probability } p_{j+1} \\
            \mathbf{Y}_{i-1,j}^\mathrm{U} + \left(1 + \epsilon / 2 \right) \mathbf{a}_{j+1} \mathbf{a}_{j+1}^\top & \text{otherwise.}
        \end{array}
    \right. \nonumber \\
    &= \mathbf{Y}_{i-1,j+1}^\mathrm{U} \label{equation:optimal3}
\end{align}
Therefore, it follows from \eqref{equation:optimal1}, \eqref{equation:optimal2}, and \eqref{equation:optimal3} that \cref{equation:optimal_compicated} is equal to 
\begin{align}
    \mathbb{E} \left[ \mathbf{a}_i^\top \left( \mathbf{Y}_{i-1,j+1}^{\mathrm{U}} \right)^+ \mathbf{a}_i - \mathbf{a}_i^\top \left( \mathbf{Y}_{i-1,j}^\mathrm{U} \right)^+ \mathbf{a}_i \right]
    \le \left( \mathbf{a}_i^\top \left( \mathbf{Y}_{i-1,j}^\mathrm{U} \right)^+ \mathbf{u}_{j+1}^\mathrm{U} \right)^2 \leq 0.
\end{align}
Thus the monotonicity \eqref{equation:monotonicity1} holds. The other monotonicity \eqref{equation:monotonicity2} holds similarly.
\end{proof}

Therefore, we have
\begin{align}
    \mathbb{E}\left[ p_i \right]
    &= c_\mathrm{U} \mathbb{E} \left[\mathbf{a}_i^\top \left( \mathbf{X}_{i-1}^\mathrm{U} \right)^+ \mathbf{a}_i \right] + c_\mathrm{L} \mathbb{E} \left[ \mathbf{a}_i^\top \left( \mathbf{X}_{i-1}^\mathrm{L} \right)^+ \mathbf{a}_i \right] \nonumber \\
    &= c_\mathrm{U} \mathbb{E} \left[\mathbf{a}_i^\top \left( \mathbf{Y}^\mathrm{U}_{i-1,i-1} \right)^+ \mathbf{a}_i \right] + c_\mathrm{L} \mathbb{E} \left[ \mathbf{a}_i^\top \left( \mathbf{Y}^\mathrm{L}_{i-1,i-1} \right)^+ \mathbf{a}_i \right] \nonumber \\
    &\le c_\mathrm{U} \mathbb{E} \left[\mathbf{a}_i^\top \left( \mathbf{Y}^\mathrm{U}_{i-1,0} \right)^+ \mathbf{a}_i \right] + c_\mathrm{L} \mathbb{E} \left[ \mathbf{a}_i^\top \left( \mathbf{Y}^\mathrm{L}_{i-1,0} \right)^+ \mathbf{a}_i \right] \nonumber \\
    &= \frac{2 c_\mathrm{U}}{\epsilon} \tau_i^{\mathbf{A}_{i-1}} \left(\mathbf{A}\right) + \frac{2 c_\mathrm{L}}{\epsilon} \tau_i^{\mathbf{A}_{i-1}} \left(\mathbf{A}\right) \nonumber \\
    &= \frac{10}{\epsilon^2} \tau_i^{\mathbf{A}_{i-1}} \left(\mathbf{A}\right). \nonumber
\end{align}
As the expected value of the sum of random variables is equal to the sum of their expectations, the row-size of $\widetilde{\mathbf{A}}$ is $\mathrm{O} \left( \epsilon^{-2} \sum_i \tau_i^{\mathbf{A}_{i-1}} \left(\mathbf{A}\right) \right)$ in expectation.
\end{proof}

\begin{proof}[Proof of \cref{theorem:approximation_size_in_online_row_sampling2}]
From \cref{lemma:optimal_appendix}, it suffices to prove $\sum_i \tau_i^{\mathbf{A}_{i-1}} \left(\mathbf{A}\right) = \mathrm{O} \left( r \log \mu(\mathbf{A}) \right)$.
Define index sets $U, V$ as $U = \{ i \mid \mathbf{a}_i \perp \mathrm{Ker} (\mathbf{A}_{i-1}) \}$, $V = \{ i \mid \mathbf{a}_i \not \perp \mathrm{Ker} (\mathbf{A}_{i-1}) \}$.
  If $\mathbf{a}_i \perp \mathrm{Ker}(\mathbf{\mathbf{A}}_{i-1})$, by \cref{lemma:Det}, we have
  \begin{align}
      \mathrm{Det} \left( \mathbf{A}_i^\top \mathbf{A}_i \right)
      &= \mathrm{Det} \left( \mathbf{A}_{i-1}^\top \mathbf{A}_{i-1} \right) \left(1 + \mathbf{a}_i^\top \left( \mathbf{A}_{i-1}^\top \mathbf{A}_{i-1} \right) ^+ \mathbf{a}_i \right) \nonumber \\
      &\geq \mathrm{Det} \left( \mathbf{A}_{i-1}^\top \mathbf{A}_{i-1} \right) \left(1 + \mathbf{a}_i^\top \left( \mathbf{A}_i^\top \mathbf{A}_i \right) ^+ \mathbf{a}_i \right) \nonumber \\
      &\geq \mathrm{Det} \left( \mathbf{A}_{i-1}^\top \mathbf{A}_{i-1} \right) \left(1 + \tau_i^{\mathbf{A}_{i-1}} \left(\mathbf{A}\right) \right) \nonumber \\
      &\geq \mathrm{Det} \left( \mathbf{A}_{i-1}^\top \mathbf{A}_{i-1} \right) \exp \left( \frac{\tau_i^{\mathbf{A}_{i-1}} \left(\mathbf{A}\right)}{2} \right). \nonumber
  \end{align}
  Otherwise, by \cref{lemma:Detequation:ieq}, 
  \begin{displaymath}
    \mathrm{Det} \left( \mathbf{A}_i^\top \mathbf{A}_i \right) \geq \lambda_{\mathrm{min}} \left( \mathbf{A}_i^\top \mathbf{A}_i \right) \mathrm{Det} \left( \mathbf{A}_{i-1}^\top \mathbf{A}_{i-1} \right).
  \end{displaymath}
  Combining the two inequalities, we have
  \begin{displaymath}
    \mathrm{Det} \left( \mathbf{A}^\top \mathbf{A} \right) \geq \exp \left( \sum_{i \in U} \frac{\tau_i^{\mathbf{A}_{i-1}} \left(\mathbf{A}\right)}{2} \right) \prod_{i \in V} \lambda_{\min} \left( \mathbf{A}_i^\top \mathbf{A}_i \right) \mathrm{Det} \left( \mathbf{O} \right).
  \end{displaymath}
  Note that $\mathrm{Det} \left(\mathbf{A}^\top \mathbf{A} \right) \leq \left( \| \mathbf{A} \|_2^2 \right)^r$ holds. 
  Taking the logarithm of both sides, we have
  \begin{align}
    r \log \left( \| \mathbf{A} \|^2_2 \right)
    &\geq \sum_{i \in U} \frac{\tau_i^{\mathbf{A}_{i-1}} \left(\mathbf{A}\right)}{2} + \sum_{i \in V} \log \left( \lambda_{\min} \left( \mathbf{A}_i^\top \mathbf{A}_i \right) \right) \nonumber \\
    &\geq \sum_{i \in U} \frac{\tau_i^{\mathbf{A}_{i-1}} \left(\mathbf{A}\right)}{2} + r \log \left( \underset{1 \le i \le n}{\min}\  \lambda_{\min} \left( \mathbf{A}_i^\top \mathbf{A}_i \right) \right). \nonumber
  \end{align}
  as $| V | = \mathrm{rank} (\mathbf{A}) = r$.
  Since $l_i = 1$ for all $i \in V$, we obtain
  \begin{align}
    \sum_{i=1}^n \tau_i^{\mathbf{A}_{i-1}} \left(\mathbf{A}\right)
      &\leq 2 r \log \left( \frac{\| \mathbf{A} \|^2_2}{\underset{1 \le i \le n}{\min}\  \lambda_{\min} \left( \mathbf{A}_i^\top \mathbf{A}_i \right)} \right) + r. \nonumber \\
    &= \mathrm{O} \left( r \log \mu (\A) + r \right). \nonumber
  \end{align}
\end{proof}

\subsection{Online Random Order Setting}

\begin{theorem}
\label{theorem:optimal_random_order_streams}
    Let $\epsilon \in (0, 1)$ be an error parameter and $\X$ be a family of $n$ vectors in $\mathbb{R}^d$.
    \cref{algorithm:optimal} constructs in the online random order setting $\A \sim \mathcal{A}(\X)$ which satisfies the following with high probability:
    The algorithm returns a $(1 \pm \epsilon)$-spectral approximation for $\mathbf{A}$ with $\mathrm{O}\left( d \epsilon^{-2} \log n \right)$ rows in expectation.
\end{theorem}

\begin{proof}
We use the same notation as \cref{algorithm:scaled_sampling}. From \cref{eq:tau_ij} in \cref{lemma:analysis_random_order_streams2}, we have
\begin{align}
    \sum_{i=1}^{\alpha} \sum_{j=\left(2^i-1\right)K + 1}^{\left(2^{i+1} - 1 \right)K} \widehat{\tau}_{i,j} = \mathrm{O} \left( d \log n \right) \label{equation:tau_ij2}
\end{align}
We will show that $\tau_j^{\mathbf{A}_{j-1}} \left( \A \right) \le \widehat{\tau}_{i,j}$ holds for $1 \le i \le \alpha$ and $\left( 2^i - 1 \right)K + 1 \le j \le \left( 2^{i+1} - 1 \right)K$, which implies that $\sum \tau_j^{\mathbf{A}_{j-1}} \left( \A \right) = \mathrm{O} \left( d \log n \right)$ by \cref{equation:tau_ij2}. If $\mathbf{a}_j \perp \mathrm{Ker} \left( \mathbf{M}_{i-1} \right)$, we have
\begin{displaymath}
    \mathbf{a}_j^\top \left( \mathbf{A}_j^\top \mathbf{A}_j \right)^+ \mathbf{a}_j
    = \min \left( \mathbf{a}_j^\top \left( \mathbf{A}_j^\top \mathbf{A}_j \right)^+ \mathbf{a}_j, 1 \right)
    \le \min \left( \mathbf{a}_j^\top \left( \mathbf{M}_{i-1}^\top \mathbf{M}_{i-1} \right)^+ \mathbf{a}_j, 1 \right)
    = \widehat{\tau}_{i,j}.
\end{displaymath}
Otherwise $\mathbf{a}_j^\top \left( \mathbf{A}_j^\top \mathbf{A}_j \right)^+ \mathbf{a}_j = 1 = \widehat{\tau}_{i,j}$. Hence we have $\sum \tau_j^{\mathbf{A}_{j-1}} \left( \A \right) = \mathrm{O} \left( d \log n \right)$.
Combining it with \cref{lemma:optimal_appendix}, $\widetilde{\mathbf{A}}$ has $\mathrm{O} \left( d \epsilon^{-2} \log n \right)$ rows in expectation.
\end{proof}

\section{Properties of Pseudo-Inverse}\label{app:pseudoinverse}
In this section, we prove two lemmas regarding positive semidefinite ordering of matrices used frequently in the paper.
\begin{lemma}
\label{lemma:appendixB_ralation1}
Let $\mathbf{A}, \mathbf{B} \in \mathbb{R}^{d \times d}$ be $\mathrm{PSD}$ matrices. For any matrix $\mathbf{V} \in \mathbb{R}^{d \times d}$, we have
\begin{displaymath}
    \mathbf{A} \preceq \mathbf{B} \Rightarrow \mathbf{V}^\top \mathbf{A} \mathbf{V} \preceq \mathbf{V}^\top \mathbf{B} \mathbf{V}.
\end{displaymath}
\end{lemma}
\begin{proof}
The lemma follows from:
\begin{align}
    \mathbf{A} \preceq \mathbf{B} &\Leftrightarrow \underset{\mathbf{x} \in \mathbb{R}^d}{\min}\  \mathbf{x}^\top \left( \mathbf{B} - \mathbf{A} \right) \mathbf{x} \ge 0 \nonumber \\
    &\Rightarrow \underset{\mathbf{y} \in \mathbb{R}^d}{\min}\  \left( \mathbf{V}\mathbf{y} \right)^\top \left( \mathbf{B} - \mathbf{A} \right) \mathbf{V}\mathbf{y} \ge 0 \hspace{7pt} \left( \because \mathbf{x} \text{ is written as } \mathbf{V} \mathbf{y} \right) \nonumber \\
    &\Leftrightarrow \mathbf{V}^\top \mathbf{A} \mathbf{V} \preceq \mathbf{V}^\top \mathbf{B} \mathbf{V}. \nonumber
\end{align}
\end{proof}
\begin{lemma}
\label{lemma:pseudoinverse}
Let $\mathbf{A}, \mathbf{B} \in \mathbb{R}^{d \times d}$ be $\mathrm{PSD}$ matrices. We have
\begin{displaymath}
    \A \preceq \B \Rightarrow \mathbf{x}^\top \B^+ \mathbf{x} \le \mathbf{x}^\top \A^+ \mathbf{x} \hspace{7pt} \text{for any $\mathbf{x} \in \mathrm{Im} \left( \A \right)$}.
\end{displaymath}
\end{lemma}
\begin{proof}
From \cref{lemma:appendixB_ralation1}, we get
\begin{displaymath}
    \mathbf{B}^{+/2} \mathbf{A} \mathbf{B}^{+/2} \preceq \mathbf{B}^{+/2} \mathbf{B} \mathbf{B}^{+/2}.
\end{displaymath}
If $\mathbf{x} \in \mathrm{Im} \left( \B \right)$, we see $\B^{1/2} \B^{+/2} \mathbf{x} = \mathbf{x}$. Hence we have
\begin{displaymath}
    \mathbf{x}^\top \mathbf{B}^{+/2} \mathbf{A} \mathbf{B}^{+/2} \mathbf{x} \le \mathbf{x}^\top \mathbf{x} \hspace{7pt} \text{for any $\mathbf{x} \in \mathrm{Im} \left( \B \right)$}.
\end{displaymath}
We can rewrite it as follows:
\begin{align}
    \underset{\mathbf{x} \in \mathrm{Im} \left( \B \right), \| \mathbf{x} \|_2 = 1}{\max} \mathbf{x}^\top \B^{+/2}\A \B^{+/2} \mathbf{x} \le 1. \label{inequality:appendixB_lemma}
\end{align}
We will show
\begin{align}
    \underset{\mathbf{y} \in \mathrm{Im} \left( \A \right), \| \mathbf{y} \|_2 = 1}{\max} \mathbf{y}^\top \A^{1/2} \B^+ \A^{1/2} \mathbf{y} \le 1. \label{inequality:appendixB_target}
\end{align}
The inequality \eqref{inequality:appendixB_target} is equivalent to the desired inequality from the following relation:
\begin{align}
    \mathbf{y}^\top \A^{1/2} \B^+ \A^{1/2} \mathbf{y} \le \mathbf{y}^\top \mathbf{y} \hspace{7pt} \text{for any $\mathbf{y} \in \mathrm{Im} \left( \A \right)$}
    \Leftrightarrow \mathbf{z}^\top \B^+ \mathbf{z} \le \mathbf{z}^\top \A^+ \mathbf{z} \hspace{7pt} \text{for any $\mathbf{z} \in \mathrm{Im} \left( \A \right)$}. \nonumber
\end{align}
Let $\lambda_1, \lambda_2, \dots, \lambda_d$ be eigenvalues of $\A^{1/2} \B^+ \A^{1/2}$ and ${\mathbf{u}_i}$ be orthonormal eigenvectors such that for all $i$ $\mathbf{u}_i$ is corresponding to $\lambda_i$. 
Define the $d$ dimensional vector $\mathbf{v}_i \overset{\mathrm{def}}{=} \B^{+/2}\A^{1/2}\mathbf{u}_i$ for $1 \le i \le n$.

Then, we obtain for $1 \le i \le n$
\begin{align}
    \A^{1/2}\B^{+}\A^{1/2} \mathbf{u}_i = \lambda_i \mathbf{u}_i &\Leftrightarrow \A^{1/2} \B^{+/2} \mathbf{v}_i = \lambda_i \mathbf{u}_i \nonumber \\
    &\Rightarrow \B^{+/2} \A \B^{+/2} \mathbf{v}_i = \lambda_i \B^{+/2} \A^{1/2} \mathbf{u}_i \nonumber \\
    &\Leftrightarrow \B^{+/2} \A \B^{+/2} \mathbf{v}_i = \lambda_i \mathbf{v}_i \label{equation:appendixB_eigenvalue}.
\end{align}
We define $\mathbf{y}^*$ as follows:
\begin{align}
    \mathbf{y}^* \overset{\mathrm{def}}{=} \underset{\mathbf{y} \in \mathrm{Im} \left( \A \right), \| \mathbf{y} \|_2 = 1}{\mathrm{arg} \max} \mathbf{y}^\top \A^{1/2} \B^+ \A^{1/2} \mathbf{y} \nonumber
\end{align}
and assume that $\mathbf{y}^*$ is represented as $\mathbf{y}^* = \sum_i q_i \mathbf{u}_i \nonumber$ with $q_i \in \mathbb{R}$.
We get
\begin{align}
    \mathbf{y}^{*\top} \A^{1/2} \B^+ \A^{1/2} \mathbf{y}^*
    &= \left( \sum_{i=1}^n q_i \mathbf{u}_i \right)^\top \A^{1/2} \B^+ \A^{1/2} \left( \sum_{i=1}^n q_i \mathbf{u}_i \right) \nonumber \\
    &= \sum_{i=1}^n  q_i^2 \mathbf{u}_i^\top \A^{1/2} \B^+ \A^{1/2} \mathbf{u}_i \nonumber \\
    &= \sum_{i=1}^n  q_i^2 \mathbf{v}_i^\top \mathbf{v}_i \nonumber \\
    &= \sum_{i=1}^n \frac{q_i^2}{\lambda_i} \mathbf{v}_i^\top \B^{+/2} \A \B^{+/2} \mathbf{v}_i \nonumber \\
    &= \left( \sum_{i=1}^n \frac{q_i}{\sqrt{\lambda_i}} \mathbf{v}_i \right)^\top \B^{+/2} \A \B^{+/2} \left( \sum_{i=1}^n \frac{q_i}{\sqrt{\lambda_i}} \mathbf{v}_i \right) \label{equation:appendixB_middle}
\end{align}
since $\mathbf{v}_i^\top \B^{+/2} \A \B^{+/2} \mathbf{v}_j = \lambda_j \mathbf{v}_i^\top \mathbf{v}_j = \lambda_j \mathbf{u}_i^\top \A^{1/2} \B^{+} \A^{1/2} \mathbf{u}_j = \lambda_j^2 \mathbf{u}_i^\top \mathbf{u}_j = 0$ for $i \neq j$.
Define the vector $\mathbf{t} \overset{\mathrm{def}}{=} \sum_{i=1}^n \left( q_i / \sqrt{\lambda_i} \right) \mathbf{v}_i$.
We have $\mathbf{t} \in \mathrm{Im}\left( \B \right)$ and
\begin{align}
    \| \mathbf{t} \|_2 =
    \left( \sum_{i=1}^n \frac{q_i}{\sqrt{\lambda_i}} \mathbf{v}_i \right)^\top  \left( \sum_{i=1}^n \frac{q_i}{\sqrt{\lambda_i}} \mathbf{v}_i \right) = \sum_{i=1}^n \frac{q_i^2}{\lambda_i} \mathbf{u}_i^\top \A^{1/2} \B^{+} \A^{1/2} \mathbf{u}_i = \sum_{i=1}^n \frac{q_i^2}{\lambda_i} \lambda_i \mathbf{u}_i^\top \mathbf{u}_i = \| \mathbf{y}^* \|_2 = 1.
\end{align}
Combining it with the inequality \eqref{inequality:appendixB_lemma} and \cref{equation:appendixB_middle}, we obtain \cref{inequality:appendixB_target} as follows:
\begin{align}
    \underset{\mathbf{y} \in \mathrm{Im} \left( \A \right), \| \mathbf{y} \|_2 = 1}{\max} \mathbf{y}^\top \A^{1/2} \B^+ \A^{1/2} \mathbf{y}
    &= \mathbf{y}^{*\top} \A^{1/2} \B^+ \A^{1/2} \mathbf{y}^* \nonumber \\
    &= \mathbf{t}^\top \B^{+/2} \A \B^{+/2} \mathbf{t} \nonumber \\
    &\le \underset{\mathbf{x} \in \mathrm{Im} \left( \B \right), \| \mathbf{x} \|_2 = 1}{\max} \mathbf{x}^\top \B^{+/2}\A \B^{+/2} \mathbf{x} \nonumber \\
    &\le 1. \nonumber
\end{align}
\end{proof}

\bibliographystyle{alpha}
\bibliography{references}

\end{document}